\documentclass[10pt]{amsart}

\usepackage{amsmath}
\usepackage{amssymb}
\usepackage{bm}
\usepackage{graphicx}
\usepackage{psfrag}
\usepackage{color}
\usepackage{hyperref}
\hypersetup{colorlinks=true, linkcolor=blue, citecolor=magenta, urlcolor=wine}
\usepackage{url}
\usepackage{algpseudocode}
\usepackage{fancyhdr}
\usepackage{mathtools}
\usepackage{tikz-cd}
\usepackage{xy}
\input xy
\xyoption{all}
\usepackage{stmaryrd}

\newtheorem*{maintheorem*}{Main Theorem}
\newtheorem{theorem}{Theorem}[section]
\newtheorem*{theorem*}{Main Theorem}

\newtheorem{prop}[theorem]{Proposition}
\newtheorem{conj}[theorem]{Conjecture}
\newtheorem{lemma}[theorem]{Lemma}
\newtheorem{cor}[theorem]{Corollary}

\theoremstyle{definition}
\newtheorem{definition}[theorem]{Definition}
\newtheorem{remark}[theorem]{Remark}
\newtheorem{example}[theorem]{Example}
\numberwithin{equation}{section}

\newcommand{\nn}{\mathbb{N}}
\newcommand{\pp}{\mathbb{P}}
\newcommand{\qq}{\mathbb{Q}}
\newcommand{\rr}{\mathbb{R}}
\newcommand{\zz}{\mathbb{Z}}

\newcommand{\lcm}{\text{lcm}}
\providecommand\ldb{\llbracket}
\providecommand\rdb{\rrbracket}

\newcommand{\supp}{\textsf{supp}}

\keywords{information semialgebras, Puiseux semirings, upper triangular matrices, factorizations, atomicity, ACCP, BF-monoids, FF-monoids}

\subjclass[2010]{Primary: 15A23, 20M13; Secondary: 16Y60, 11Y05}

\thanks{The authors would like to thank Alfred Geroldinger and the Karl Franzens University of Graz for hosting them during the summer semester of 2019, where much of this work was carried out. The second author was partially supported by the NSF award DMS-1903069.}

\begin{document}
	
	\mbox{}
	\title{Factorizations in upper triangular matrices over information semialgebras}
	
	\author{Nicholas R.~Baeth}
	\address{Department of Mathematics\\Franklin and Marshall College\\Lancaster, PA 17604}
	\email{nicholas.baeth@fandm.edu}
	
	\author{Felix Gotti}
	\address{Department of Mathematics\\University of Florida\\Gainesville, FL 32611}
	\email{felixgotti@ufl.edu}
	
	\date{\today}
	
\begin{abstract}
An integral domain (or a commutative cancellative monoid) is atomic if every nonzero nonunit element is the product of irreducibles, and it satisfies the ACCP if every ascending chain of principal ideals eventually stabilizes. The interplay between these two properties has been investigated since the 1970s. An atomic domain (or monoid) satisfies the finite factorization property (FFP) if every element has only finitely many factorizations, and it satisfies the bounded factorization property (BFP) if for each element there is a common bound for the number of atoms in each of its factorizations. These two properties have been systematically studied since being introduced by Anderson, Anderson, and Zafrullah in 1990. Noetherian domains satisfy the BFP, while Dedekind domains satisfy the FFP. It is well known that for commutative cancellative monoids (in particular, integral domains) FFP $\Rightarrow$ BFP $\Rightarrow$ ACCP $\Rightarrow$ atomic. For $n \ge 2$, we show that each of these four properties transfers back and forth between an information semialgebras $S$ (i.e., a commutative cancellative semiring) and their multiplicative monoids $T_n(S)^\bullet$ of $n \times n$ upper triangular matrices over~$S$. We also show that a similar transfer behavior takes place if one replaces $T_n(S)^\bullet$ by the submonoid $U_n(S)$ consisting of unit triangular matrices. As a consequence, we find that the chain FFP $\Rightarrow$ BFP $\Rightarrow$ ACCP $\Rightarrow$ atomic also holds for the classes comprising the noncommutative monoids $T_n(S)^\bullet$ and $U_n(S)$. Finally, we construct various rational information semialgebras to verify that, in general, none of the established implications is reversible.
\end{abstract}

\bigskip

\maketitle

%%%%%%%%%%%%%%
\section{Introduction}
\label{sec:intro}

A factorization of an element in a commutative cancellative monoid is a representation of that element as a formal product of atoms (i.e., irreducible elements). When every nonunit element has such a representation, the monoid is called atomic and, additionally, if such a representation is unique, the monoid is called a unique factorization monoid (or a UFM). The monoid is called a finite factorization monoid (or an FFM) if every nonunit element has only finitely many factorizations, and it is called a bounded factorization monoid (or a BFM) if for each nonunit element there is a common bound for the number of atoms (counting repetitions) in each of its factorizations. In addition, the monoid is called a half-factorial monoid (or an HFM) if any two factorizations of the same nonunit element have the same number of atoms (counting repetitions). These subclassifications of atomic monoids, as well as the ACCP property (every ascending chain of principal ideals stabilizes), have been well studied over the past half century. Atomic monoids, monoids satisfying the ACCP, and HFMs have been systematically studied since the 1970s, while FFMs and BFMs have been studied since they were introduced in 1990~\cite{AAZ90} in the context of integral domains. In that paper it was shown that if $M$ is the multiplicative monoid of an integral domain, then each of the implications shown in Diagram~(\ref{equation:anderson}) holds. Moreover, in the same, examples are given to show that none of the implications are, in general, reversible.

\smallskip

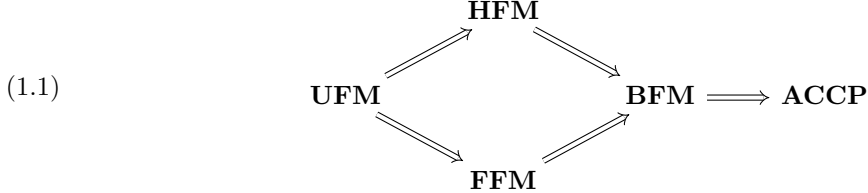
\begin{equation}\label{equation:anderson}
	\begin{tikzcd}
		& \textbf{HFM} \arrow[dr, Rightarrow] & \\
		\textbf{UFM} \arrow[ur, Rightarrow] \arrow[dr, Rightarrow] & &  \textbf{BFM} \arrow[r, Rightarrow] & \textbf{ACCP} \\
		& \textbf{FFM} \arrow[ur, Rightarrow] & 
	\end{tikzcd}
\end{equation}

\smallskip

Factorization theory has been significantly less developed in noncommutative settings, with much of the early work focusing primarily on characterizing when a given monoid is a UFM (see, for instance, \cite[Chapter~3]{pC85} and~\cite{CS03}). However, in recent years there has been more consideration given to factorizations in noncommutative rings and monoids. In particular, many factorization tools from commutative monoids and domains have been used in and adapted to noncommutative algebraic structures, including rings of upper triangular nonnegative matrices~\cite{CZL15}, maximal orders in central simple algebras~\cite{dS13}, noncommutative finite factorization domains~\cite{BHL17}, noncommutative Krull monoids~\cite{aG13}, and small cancellative categories~\cite{BS15}.

In many cases, factorization aspects of noncommutative algebraic structures are conveniently investigated through the lens of transfer homomorphisms to easier-to-understand commutative objects. For example, in~\cite{BBG} the noncommutative monoids $T_n(R)$ of upper triangular matrices over a commutative ring $R$ are studied using transfer homomorphisms to products of commutative cancellative monoids, and in~\cite{BS15} various arithmetic aspects of noncommutative cancellative monoids are studied using transfer homomorphisms to their reduced abelianizations. By contrast, it was proved in~\cite{BS20} that for a reduced information semialgebra $S$, there are no such transfer homomorphisms from the monoid $T_n(S)^\bullet$ of regular elements of $T_n(S)$. Using other approaches, however, it was shown in~\cite{BS20} that $T_n(S)^\bullet$ is atomic, after which some arithmetical invariants were computed.

The present paper can be thought of as a continuation of~\cite{BS20}. However, our primary goal here is to provide a more fundamental set of results on the atomicity of the noncommutative monoid $T_n(S)^\bullet$ as well as its submonoid $U_n(S)$ consisting of unit upper triangular matrices. We characterize when they are FFMs or BFMs, determine when they satisfy the ACCP, and argue that $T_n(S)^\bullet$ is almost never a HFM. To do so, we prove that each of these properties, save half-factoriality, transfers back and forth from the monoids $T_n(S)^\bullet$ and $U_n(S)$ to both the additive and multiplicative monoids $(S,+)$ and $(S^\bullet,\cdot)$. In particular, we give a set of implications analogous to those in Diagram~(\ref{equation:anderson}) but for both $T_n(S)^\bullet$ and $U_n(S)$. By considering Puiseux information semialgebras (i.e., semialgebras contained in the nonnegative cone of rational numbers), we illustrate that, as in the case of commutative monoids, each of the implications is not reversible in general.

%\smallskip

This paper is structured as follows.  In Section~\ref{sec:prelim} we introduce the main objects of study and other related definitions and notation. Then in Section~\ref{sec:Puiseux information semialgebras} we introduce the notion of a Puiseux information semialgebra and explore some of their atomic aspects only far enough to use them as our primary source of examples later in Section~\ref{sec:upper triangular matrices}. Our main results are contained in Section~\ref{sec:upper triangular matrices}. We summarize those results here. If $S$ is a reduced information semialgebra and $\blacksquare\in \{\text{FFM, BFM, ACCP, atomic}\}$, then we say that $S$ is bi-$\blacksquare$ provide that both $(S,+)$ and $(S^\bullet,\cdot)$ are $\blacksquare$. Further relevant terminology can be found in Sections 2--4.

\smallskip

\begin{theorem*}
Let $S$ be an information semialgebra. If $n\geq 2$, then $T_n(S)^\bullet$ is never half-factorial. Also, each implication in the following diagram holds. 
\begin{equation*} %\label{diag:full atomic diagram}
	\begin{tikzcd}
		S \text{ is a \textbf{bi-FFM}} \arrow[r, Rightarrow] \arrow[d, Leftrightarrow] & S \text{ is a \textbf{bi-BFM}} \arrow[r, Rightarrow] \arrow[d, Leftrightarrow] &  S \text{ is \textbf{bi-ACCP}} \arrow[r, Rightarrow] \arrow[d, Leftrightarrow] & S \text{ is \textbf{bi-atomic}} \arrow[d, Leftrightarrow] \\
		T_n(S)^\bullet \text{ is an \textbf{FFM}} \arrow[r, Rightarrow] & T_n(S)^\bullet \text{ is a \textbf{BFM}} \arrow[r, Rightarrow] & T_n(S)^\bullet \text{ is \textbf{ACCP}} \arrow[r, Rightarrow] & T_n(S)^\bullet \text{ is \textbf{atomic}}
	\end{tikzcd}
\end{equation*}
Moreover, none of the horizontal inclusions is, in general, reversible. 
\end{theorem*}

\bigskip

%%%%%%%%%%%%%%%
\section{Fundamentals}
\label{sec:prelim}

In this section we introduce the relevant concepts pertaining to commutative monoids and factorization theory required in later sections. For a thorough treatment on commutative monoids, see~\cite{pG01}; for atomic monoids and factorization theory, see~\cite{GH06}.

\medskip

%%%%%%%%%%%%%%%%%%%
\subsection{General Notation}
We begin by introducing the general notation we shall be using throughout this paper. We let $\nn = \{1,2,\dots\}$ denote the set of positive integers and set $\nn_0 = \nn \cup \{0\}$. In addition, we let $\pp$ denote the set consisting of all prime numbers. For $a,b \in \zz$ we let $\ldb a,b \rdb$ denote the set of integers between $a$ and $b$, i.e., 
\[
	\ldb a,b \rdb := \{z \in \zz : a \le z \le b\}.
\]
In addition, for $X \subseteq \rr$ and $r \in \rr$, we set
\[
	X_{\ge r} := \{x \in X : x \ge r\}
\]
and we use the notations $X_{> r}, X_{\le r}$, and $X_{< r}$ analogously. If $q \in \qq_{> 0}$, then we call the unique $n,d \in \nn$ such that $q = n/d$ and $\gcd(n,d)=1$ the \emph{numerator} and \emph{denominator} of $q$ and denote them by $\mathsf{n}(q)$ and $\mathsf{d}(q)$, respectively. Finally, for $Q \subseteq \qq_{>0}$, we set
\[
	\mathsf{n}(Q) := \{\mathsf{n}(q) : q \in Q\} \quad \text{ and } \quad \mathsf{d}(Q) := \{\mathsf{d}(q) : q \in Q\}.
\]

\medskip

%%%%%%%%%%%%%%
\subsection{Monoids} A monoid is defined to be a semigroup with identity, and we tacitly assume that all monoids here are cancellative, though not necessarily commutative. Let $M$ be a monoid. The set of invertible elements of $M$ is a group, which we denote by $U(M)$. The monoid $M$ is called \emph{reduced} if $|U(M)|=1$. An element $a \in M \! \setminus \! U(M)$ is called an \emph{atom} if whenever $a = uv$ for $u,v \in M$, either $u \in U(M)$ or $v \in U(M)$. The set of atoms of $M$ is denoted by $\mathcal{A}(M)$. The monoid $M$ is called \emph{atomic} if each non-invertible element can be written as a product of atoms, and $M$ is called \emph{antimatter} if $\mathcal{A}(M)$ is empty.

A subset $I$ of $M$ is called a \emph{left ideal} (resp., \emph{right ideal}) provided that $M \, I \subseteq I$ (resp., $I \, M \subseteq I$). A (\emph{two-sided}) \emph{ideal} of $M$ is a subset that is both a left and a right ideal. For each $x \in M$, the left ideal $Mx$ and the right ideal $xM$ are called \emph{principal}. As in ring theory, we say that $M$ satisfies the \emph{ACCP} (resp., \emph{left ACCP} or \emph{right ACCP}) if each ascending chain of principal ideals (resp., left ideals or right ideals) eventually stabilizes. If $M$ satisfies both the left ACCP and the right ACCP, then $M$ is atomic (see~\cite[Proposition~3.1]{dS13}).

If $x \in M$ and $x = a_1 \dots a_\ell$ for some $\ell \in \nn$ and $a_1, \dots, a_\ell \in \mathcal{A}(M)$, then $\ell$ is called a (\emph{factorization}) \emph{length} of $x$. The \emph{set of lengths} of $x \in M \! \setminus \! U(M)$, denoted by $\mathsf{L}(x)$, is the set of all possible lengths of~$x$. In addition, we define $\mathsf{L}(u) := \{0\}$ for all $u \in U(M)$. Clearly, $M$ is atomic if and only if $\mathsf{L}(x) \neq \emptyset$ for all $x \in M$. An atomic monoid $M$ is called a \emph{BFM} (or a \emph{bounded factorization monoid}) provided that $|\mathsf{L}(x)| < \infty$ for all $x \in M$.
\medskip

%%%%%%%%%%%%%%%
\subsection{Semirings} A triple $(S,+, \cdot)$, where $S$ is a nonempty set and both $+$ and $\cdot$ are binary operations on $S$ (called \emph{addition} and \emph{multiplication}) is said to be a \emph{semiring} if the following conditions hold:
\begin{enumerate}
	\item[1.] $(S,+)$ is a commutative monoid with identity element denoted by $0$;
	\smallskip
	
	\item[2.] $(S, \cdot)$ is a semigroup with identity element denoted by $1$;
	\smallskip
	
	\item[3.] multiplication distributes over addition;
	\smallskip
	
	\item[4.] $0 \cdot x = 0$ for all $x \in S$.
\end{enumerate}
The semiring $S$ is said to be \emph{commutative} if $(S, \cdot)$ is a commutative semigroup. We let $S^\times$ denote the group of invertible elements of $(S, \cdot)$. As in ring theory, an element $x \in S$ is called a \emph{left zero-divisor} (resp., a \emph{right zero-divisor}) provided that there exists $y \in S \setminus \{0\}$ such that $xy=0$ (resp., $yx=0$). In addition, $x \in S$ is called \emph{regular} if it is neither a left zero-divisor nor a right zero-divisor. We let~$S^\bullet$ denote the multiplicative semigroup consisting of all regular elements of $S$.

For $n \in \nn$ and a semiring $S$, let $T_n(S)$ denote the set consisting of all $n \times n$ upper triangular matrices with entries in $S$. It is clear that $T_n(S)$ is a semiring with respect to the usual addition and multiplication of matrices, where the identity element of the semigroup $(T_n(S), \cdot)$ is $I_n$, the identity matrix. Note that the multiplicative semigroup $T_n(S)^\bullet$ is cancellative and, therefore, a monoid. For $A \in T_n(S)$ and $i,j \in \ldb 1,n \rdb$, we let $A_{ij}= [A]_{ij}$ denote the entry of $A$ in position $(i,j)$, that is, in the row $i$ and column $j$. Provided that $i \le j$, we let $E_{ij}$ denote the matrix of $T_n(S)$ having a $1$ in position $(i,j)$ and $0$'s everywhere else. A matrix $A \in T_n(S)$ is called \emph{unit triangular} if $A_{ii} = 1$ for every $i \in \ldb 1,n \rdb$, and the multiplicative monoid consisting of all unit triangular matrices of $T_n(S)$ is denoted here by $U_n(S)$. The monoids $T_n(S)^\bullet$ and $U_n(S)$ are the primary focus of this paper.

For a semiring $S$, the additive monoid $(S,+)$ may not be reduced and the multiplicative semigroup $(S^\bullet, \cdot)$ may be neither commutative nor cancellative. However, these two properties always hold for the semirings we shall be studying here under the term `information semialgebras'.

\begin{definition} \label{def:information semialgebras}
	We say that a semiring $S$ is an \emph{information semialgebra} if $(S,+)$ is a reduced monoid and $(S \setminus \{0\},\cdot)$ is a commutative monoid.\footnote{The term `information algebra' was used by Kuntzmann in 1972 for reduced semirings without nonzero zero divisors.} 
\end{definition}

Let $S$ be an information semialgebra. It follows immediately from Definition~\ref{def:information semialgebras} that $S$ is a commutative semiring without nonzero zero-divisors, and so $S^\bullet = S \setminus \{0\}$. We call $(S,+)$ and $(S^\bullet, \cdot)$ the \emph{additive monoid} and the \emph{multiplicative monoid} of $S$, respectively, and we say that $S$ is \emph{reduced} if its multiplicative monoid is reduced.

For commutative semirings $S$ and $T$, a map $\phi \colon S \to T$ is called a \emph{homomorphism of semirings} provided that $\phi \colon (S,+) \to (T,+)$ and $\phi|_{S^\bullet} \colon (S^\bullet, \cdot) \to (T^\bullet, \cdot)$ are semigroup homomorphisms satisfying that $\phi(0) = 0$ and $\phi|_{S^\bullet}(1) = 1$. A homomorphism of semirings is an \emph{isomorphism} if it is invertible. For an information semialgebra $S$, we let $\mathcal{A}_+(S)$ and $\mathcal{A}_\times(S)$ denote the set of atoms of $(S,+)$ and $(S^\bullet, \cdot)$, respectively. In~\cite{BS20}, both Proposition~2.9 and Theorem~2.10 deal with atomic information semialgebras whose additive monoids contain exactly one atom. As the next lemma indicates, there is only one such an information semialgebra up to isomorphism.

\smallskip

\begin{lemma}
For an information semialgebra $S$ the following statements are equivalent.
\begin{enumerate}
	\item[(a)] $(S,+)$ is atomic and $|\mathcal{A}_+(S)| = 1$.
	\smallskip
	
	\item[(b)] $(S,+) \cong (\nn_0,+)$, as monoids.
	\smallskip
	
	\item[(c)] $S \cong \nn_0$, as semirings.
\end{enumerate}
\end{lemma}

\begin{proof}
(a) $\Rightarrow$ (b): Let $a$ be the only atom in $\mathcal{A}(S)$. As $S$ is an atomic monoid, $S = \langle a \rangle$. Therefore it follows immediately that the assignment $n \mapsto na$ determines an isomorphism from $(\nn_0,+)$ to $(S,+)$.
\smallskip
	
(b) $\Rightarrow$ (c): Because $(S,+) \cong (\nn_0,+)$, there exists $a \in \mathcal{A}_+(S)$ such that $S = \{na : n \in \nn_0\}$, where~$na$ denotes the addition of $n$ copies of $a$ (note that $0 = 0a$). Let $e$ be the multiplicative identity element of $S$. Now write $a \cdot a = ma$ and $e = na$ for some $m, n \in \nn$. Then
\[
	na = e = e \cdot e = (na) \cdot (na) = n^2 (a \cdot a) = n^2m a.
\]
Since $(S,+)$ is free on $\{a\}$, it follows that $n = n^2m$. Then $n = 1$, which implies that $e = a$. As a consequence, $a$ is the multiplicative identity element of $S$, and so the map $\nn_0 \to S$ defined by $n \mapsto na$ is an isomorphism of semirings.
\smallskip

(c) $\Rightarrow$ (a): If $S$ is isomorphic to the information semialgebra $\nn_0$ as semirings, then $(S,+)$ is isomorphic to $(\nn_0,+)$ as monoids, from which (a) clearly follows.
\end{proof}

\smallskip

We say that an information semialgebra $S$ is \emph{bi-atomic} if both monoids $(S,+)$ and $(S^\bullet, \cdot)$ are atomic. In a similar manner, we say that $S$ is a \emph{bi-ACCP monoid} (resp., a \emph{bi-BFM}) if both $(S,+)$ and $(S^\bullet, \cdot)$ are ACCP monoids (resp., BFMs).

\medskip

%%%%%%%%%%%%%%%%%%
\subsection{Factorizations} For a set $S$ we let $\mathcal{F}^*(S)$ denote the free monoid on $S$. Let $M$ be a monoid and let $y = (u, a_1 \cdots a_m)$ and $y' = (u', a'_1 \cdots a'_n)$ be elements of $M^\times \times \mathcal{F}^*(\mathcal{A}(M))$ for some $m,n \in \nn_0$, $u,u' \in M^\times$, and $a_1, \dots, a_m, a'_1, \dots, a'_n \in \mathcal{A}(M)$. We write $y \sim y'$ if $m=n$ and either $m = 0$ or there exist $u_2, \dots, u_{m+1} \in H^\times$ satisfying that
\[
	u a_1 = u'a'_1 u_2^{-1} \quad \text{and} \quad a_j = u_j a'_j u_{j+1}^{-1}.
\]
for every $j \in \ldb 2,m \rdb$. It is not hard to verify that $\sim$ defines a congruence relation on $M^\times \times \mathcal{F}^*(\mathcal{A}(M))$ (see \cite[Section 3]{BS15}). The quotient of $M^\times \times \mathcal{F}^*(\mathcal{A}(M))$ by $\sim$ is called the \emph{rigid factorization monoid} of~$M$ and is denoted by $\mathsf{Z}^*(M)$. We denote an element $z = [(u, a_1 \cdots a_n)]_\sim \in \mathsf{Z}^*(M)$ simply by $ua_1 \cdots a_n$ and call $|z| := n$ the \emph{length} of $z$. The homomorphism $\pi \colon \mathsf{Z}^*(M) \to M$ induced by the operation of $M$ is called the \emph{rigid factorization homomorphism}. For each $x \in M$, the set
\[
	\mathsf{Z}^*(x) := \pi^{-1}(\{x\})
\]
is called the \emph{set of rigid factorizations} of $x$. It is clear that $M$ is atomic if and only if $\mathsf{Z}^*(x) \neq \emptyset$ for all $x \in M$. We say that the monoid $M$ is an \emph{FFM} (or a \emph{finite factorization monoid}) provided that $\mathsf{Z}^*(x)$ is a finite nonempty set for all $x \in M$. It follows immediately that each FFM is a BFM.

\smallskip

\begin{remark}
The notion of a rigid factorization is a recent well-behaved noncommutative analog of the notion of a factorization in commutative monoids. However, we point out that rigid factorizations do not generalize the standard definition of factorizations in commutative monoids as rigid factorizations are not equal up to permutation.
\end{remark}

\smallskip

We now recall relevant factorization terminology from the commutative setting. For a set~$P$, we let $\mathcal{F}(P)$ denote the free commutative monoid on~$P$. Each element $x \in \mathcal{F}(P)$ can be written uniquely in the form
\[
	x = \prod_{p \in P} p^{\mathsf{v}_p (x)},
\]
where $\mathsf{v}_p(x) \in \nn_0$ is the $p$-adic valuation of $x$, and we call $|x| = \sum_{p \in P} \mathsf v_p (x)$ the \emph{length} of $x$ in $\mathcal {F}(P)$. Let $M$ be a commutative monoid. The free commutative monoid $\mathcal{F}(\mathcal{A}(M_{ \text{red} }))$ is called the \emph{factorization monoid} of $M$ and is denoted by $\mathsf{Z}(M)$. In addition, the unique monoid homomorphism $\pi \colon \mathsf{Z}(M) \to M_{\text{red}}$ fixing all the atoms in the set $\mathcal{A}(M_{\text{red}})$ is called the \emph{factorization homomorphism} of $M$. For $x \in M$, the set
\[
	\mathsf{Z}(x) = \pi^{-1} (xM^\times)
\]
is called the \emph{set of factorizations} of $x$. It is clear that $\mathsf{L}(x) = \{ |z| : z \in \mathsf{Z}(x) \}$. Also, notice that $1 \in \mathsf{L}(x)$ if and only if $x \in \mathcal{A}(M)$, in which case $\mathsf{L}(x) = \{1\}$. We say that $M$ is a \emph{UFM} (or a \emph{unique factorization monoid}) if $|\mathsf{Z}(x)| = 1$ for all $x \in M$. Finally, observe that $M$ is an FFM (as defined in the noncommutative setting) if and only if $\mathsf{Z}(x)$ is finite for each $x \in M$.

\bigskip

%%%%%%%%%%%%%%%%%%%%%%%%%%%%%
\section{Puiseux Information Semialgebras}
\label{sec:Puiseux information semialgebras}

\subsection{General Facts} Many of the main examples in the next section are constructed using information semialgebras contained in $\mathbb Q_{\geq 0}$, the semialgebra of nonnegative rational numbers.  In this section we introduce information germane to their study. An additive submonoid of the nonnegative cone of rational numbers is called a \emph{Puiseux monoid}. Additive factorizations in these monoids have been widely studied in recent years (see \cite{CGG19, CGG20} and references therein). We now define Puiseux information semialgebras.

\smallskip

\begin{definition}
An information semialgebra $S$ is called a \emph{Puiseux information semialgebra} if $(S,+)$ is isomorphic to a Puiseux monoid.
\end{definition}

\smallskip

We proceed to establish some basic properties of Puiseux information semialgebras.

\smallskip

\begin{prop} \label{prop:basic aspects of PIS}
The following statements hold.
\begin{enumerate}
	\item[1.] For each Puiseux monoid $M$ containing $1$, there is at most one multiplicative operation turning~$M$ into an information semialgebra with identity element~$1$, namely the standard multiplication of rational numbers.
\vspace{2pt}
	\item[2.] Any Puiseux information semialgebra can be embedded into the nonnegative cone of $\qq$.
\vspace{2pt}
	\item[3.] If $S_1$ and $S_2$ are two Puiseux information semialgebras contained in $\qq$, then $S_1 \cong S_2$ implies that $S_1 = S_2$.
\end{enumerate}
\end{prop}

\begin{proof}
1. Let $M$ be a Puiseux monoid containing $1$, and suppose that $(M,+,*)$ is an information semialgebra with identity element $1$, where $+$ denotes the standard addition. Since $*$ distributes over~$+$, it follows immediately that $n * q = (\, \underbrace{1+ \dots + 1}_{n \text{ times}} \, ) * q = \underbrace{q + \dots + q}_{n \text{ times}} = nq$ for all $n \in \nn$ and $q \in M$. Now for all $q,r \in M_{> 0}$,
	\[
		r * q = \bigg( \frac{r}{\mathsf{d}(q)} * \mathsf{d}(q) \bigg) * q = \frac{r}{\mathsf{d}(q)} * (\mathsf{d}(q) * q) =  \frac{r}{\mathsf{d}(q)} * \mathsf{n}(q) = r q.
	\]
As a result, the operation $*$ is, indeed, the standard multiplication of rationals. Hence only the standard multiplication will turn $M$ into an information semialgebra with identity element~$1$.
\smallskip
	
2. Let $S$ be a Puiseux information semialgebra, and let $\phi \colon (S,+) \to M$ be a monoid isomorphism, where $M$ is a Puiseux monoid. After pushing the multiplication of $S$ to $M$ via $\phi$, the Puiseux monoid~$M$ turns into a Puiseux information semialgebra. It follows from the previous part that the multiplication $M$ inherits from $S$ must be the standard multiplication of rationals. So $M$ becomes a subsemiring of $\qq$. Since $M$ is a Puiseux monoid containing $1$, it is not a group and, therefore, $M \subseteq \qq_{\ge 0}$. Thus, we have embedded $S$ into the Puiseux information semialgebra $\qq_{\ge 0}$.
\smallskip
	
3. Let $\phi \colon S_1 \to S_2$ be a semiring isomorphism. In particular, $\phi$ is an isomorphism of Puiseux monoids. Then \cite[Proposition~3.2]{fG18} guarantees the existence of $q \in \qq_{> 0}$ such that $\phi(x) = qx$ for all $x \in S_1$. As $1 = \phi(1) = q$, one finds that $\phi$ is the identity map. Hence $S_1 = S_2$.
\end{proof}

\smallskip

\begin{remark}
	In virtue of Proposition~\ref{prop:basic aspects of PIS}, one can always think of a Puiseux information semialgebra as a subsemiring of $\qq_{\ge 0}$, and we shall do so from now on without explicit mention.
\end{remark}

We are primarily interested in bi-atomic information semialgebras. It is clear that $\nn_0$ is a bi-atomic Puiseux information semialgebra but, in general, Puiseux information semialgebras need not be bi-atomic. The following examples illustrate this observation.

\smallskip

\begin{example} \label{ex:conductive semiring}
Consider the Puiseux information semialgebra $S = \{0\} \cup \qq_{\ge 1}$. It follows from~\cite[Theorem~3.10]{fG17} and~\cite[Example~4.2]{GGT19} that $(S,+)$ is an atomic monoid with set of atoms $\qq \cap [1,2)$. To check that $(S^\bullet, \cdot)$ is antimatter, it suffices to observe that for any $q \in \qq_{> 1}$ one can take $n \in \nn$ large enough so that $q \cdot \frac{n}{n+1} > 1$ and then write $q = \big(q \frac{n}{n+1}\big) \big(\frac{n+1}{n}\big)$. Thus, $S$ is a Puiseux information semialgebra satisfying that $(S,+)$ is atomic while $(S^\bullet, \cdot)$ is antimatter.
\end{example}

\smallskip

\begin{example}
The Puiseux monoid $S = \langle 1/2^n : n \in \nn_0 \rangle$ is clearly closed under multiplication and, therefore, $S$ is a Puiseux information semialgebra. Since $1/2^n = 2(1/2^{n+1})$ for every $n \in \nn_0$, the monoid $(S,+)$ is antimatter. On the other hand, we shall see in Proposition~\ref{prop:atomicity of cyclic PIS} that $(S^\bullet, \cdot)$ is atomic. Hence $S$ is a Puiseux information semialgebra satisfying that $(S,+)$ is antimatter and $(S^\bullet, \cdot)$ is atomic.
\end{example}

\smallskip

Notice that the bi-atomic Puiseux information semialgebra $\nn_0$ satisfies that $|\mathcal{A}_+(\nn_0)| = 1$ while $|\mathcal{A}_\times(\nn_0)| = \infty$. This observation can be generalized as follows.

\smallskip

\begin{prop} \label{prop:size of sets of atoms for bi-atomic PIS}
If $S$ is a bi-atomic Puiseux information semialgebra, then $|\mathcal{A}_+(S)| \in \{1, \infty\}$ and $|\mathcal{A}_\times(S)| = \infty$.
\end{prop}

\begin{proof}
To prove the first statement, assume that $|\mathcal{A}_+(S)| < \infty$. Take $n \in \nn$ and $q_1, \dots, q_n \in \qq_{> 0}$ such that $(S,+) = \langle q_1, \dots, q_n \rangle$. It is clear that any element in $\mathsf{d}(S^\bullet)$ must divide $m = \lcm(\mathsf{d}(q_1), \dots, \mathsf{d}(q_n))$. For every $i \in \ldb 1,n \rdb$ the fact that $q_i^m \in S^\bullet$ ensures that $\mathsf{d}(q_i)^m = \mathsf{d}(q_i^m) \in \mathsf{d}(S^\bullet)$. As a consequence, $\mathsf{d}(q_i)^m \le m$ for every $i \in \ldb 1,n \rdb$, which implies that $\mathsf{d}(q_1) = \dots = \mathsf{d}(q_n) = 1$. Thus, $S \subseteq \nn_0$. As $\nn_0$ is contained in every Puiseux information semialgebra, $S = \nn_0$ and so $|\mathcal{A}_+(S)| = 1$. Hence $|\mathcal{A}_+(S)| \in \{1, \infty\}$.
\smallskip

To argue the second statement let us assume, by way of contradiction, that $|\mathcal{A}_\times(S)| < \infty$. Then we now consider the following two cases.
\smallskip

\noindent {\it Case 1}: $(S^\bullet, \cdot)$ is a group. Because $\nn$ is contained in $S^\bullet$, it follows that $1/n \in S^\bullet$ for every $n \in \nn$. This implies that $S^\bullet = \qq_{> 0}$. However, in this case $S = \qq_{\ge 0}$, which contradicts the fact that $(S,+)$ is an atomic monoid. 
\smallskip

\noindent {\it Case 2}: $(S^\bullet, \cdot)$ is not a group. Since $(S^\bullet, \cdot)$ is an atomic monoid that is not a group, $\mathcal{A}_\times(S) \ne \emptyset$. Then there exists $k \in \nn$ and $a_1, \ldots, a_k\in S$ such that $\mathcal{A}_\times(S) = \{a_1, \dots, a_k\}$. Suppose first that $(S^\bullet, \cdot)$ is not a reduced monoid. Taking $u \in S^\times \setminus \{1\}$, one can see that $s^n \in S^\times$ for every $n \in \nn$ and, therefore, $S^\times$ is an infinite set. As a result, $a_1S^\times$ is an infinite set of multiplicative atoms, contradicting that $|\mathcal{A}_\times(S)| < \infty$. Now suppose that $(S^\bullet, \cdot)$ is a reduced monoid.
%\textcolor{blue}{Assume, for convenience, that $a_1<a_2<\cdots<a_k$ and observe that $\mathsf {n}(a_i)\in S$ for all $i\in  \ldb 1,k \rdb$. By Dirichlet's Theorem, there are (infinitely many) primes of the form $p=\mathsf{n}(a_k)+l$ with $l\in \mathbb N$. And, since $\mathsf{n}(a_k), 1\in S$, all such primes are in $S$ and are larger than any $\mathsf{n}(a_i)$.} 
Since $\pp \subseteq S^\bullet$, there exists a prime $p$ in $S^\bullet$ such that $p \nmid \mathsf{n}(a_i)$ for any $i \in \ldb 1,k \rdb$. As $(S^\bullet, \cdot)$ is an atomic monoid, there are $n_1, \dots, n_k \in \nn_0$ such that $p = a_1^{n_1} \cdots a_k^{n_k}$. Then $p \mid \mathsf{n}(a_1)^{n_1} \cdots \mathsf{n}(a_k)^{n_k}$, contradicting the fact that $p$ is prime. Hence $|\mathcal{A}_\times(S)| = \infty$, which completes the proof.
\end{proof}

\smallskip

\begin{remark}
	To prove that $|\mathcal{A}_+(S)| \in \{1,\infty\}$ in Proposition~\ref{prop:size of sets of atoms for bi-atomic PIS} we did not appeal to the atomicity of the multiplicative monoid $(S^\bullet, \cdot)$.
\end{remark}

We proceed to study two classes consisting mostly of bi-atomic Puiseux information semialgebras having both infinitely many additive atoms and infinitely many multiplicative atoms. 

\medskip

%%%%%%%%%%%%%%%%%%%%%%%%%%%%%%%%%%%
\subsection{Cyclic Puiseux Information Semialgebras} In this section we restrict our attention to Puiseux information semialgebras that can be generated by a single element; we call them cyclic Puiseux information semialgebras.

\smallskip

\begin{definition}
For each $r \in \qq_{> 0}$, we call the Puiseux information semialgebra $S_r := \langle r^n : n \in \nn_0 \rangle$ the \emph{cyclic} (Puiseux) information semialgebra \emph{generated} by $r$.
\end{definition}

\smallskip

Let us now show that for almost all $r \in \qq_{> 0}$, the information semialgebra $S_r$ is, indeed, a reduced information semialgebra.

\smallskip

\begin{lemma} \label{lem:reducibility of S_r}
For $r \in \qq_{> 0}$ the Puiseux information semialgebra $S_r$ is reduced if and only if $r \neq 1/n$ for every $n \in \nn_{\ge 2}$.
\end{lemma}

\begin{proof}
The direct implication is an immediate consequence of the inclusion $\nn_0 \subseteq S$. For the reverse implication, suppose that $\mathsf{n}(r) > 1$ and $\mathsf{d}(r) > 1$ (note that if $\mathsf{d}(r) = 1$, then $S = \nn_0$ is clearly reduced). Take $u$ to be a multiplicative unit of $S_r^\bullet$. After replacing $u$ by $u^{-1}$ if necessary, we may assume that $u \le 1$. Write $u = \sum_{i=0}^n c_i r^i$ for $n, c_0, \dots, c_n \in \nn_0$ and $c_n \neq 0$. Since
	\[
		\frac{\mathsf{d}(r)^n}{\sum_{i=0}^n c_i \mathsf{n}(r)^{i} \mathsf{d}(r)^{n-i}} =  \bigg( \frac{1}{\mathsf{d}(r)^n} \sum_{i=0}^n c_i \mathsf{n}(r)^{i} \mathsf{d}(r)^{n-i} \bigg)^{-1} = u^{-1} \in S_r^\bullet,
	\]
	$\sum_{i=0}^n c_i \mathsf{n}(r)^{i} \mathsf{d}(r)^{n-i}$ must divide some power of $\mathsf{d}(r)$. This, along with the fact that $\gcd(\mathsf{n}(r), \mathsf{d}(r)) = 1$, enforces $c_0 \neq 0$, which implies that $u \ge 1$. Hence $u = 1$, and so $S_r$ is reduced.
\end{proof}

\smallskip 

For an atomic monoid $M$ and $x \in M$, we let $A_M(x)$ denote the set of all the atoms of $M$ dividing~$x$ and we let $D_M(x)$ denote the set of all elements of $M$ dividing $x$. The following lemma will be used in the proof of Proposition~\ref{prop:atomicity of cyclic PIS} and later in Section~\ref{sec:upper triangular matrices}.

\smallskip

%\begin{lemma}(cf. \cite[Theorem~5.1]{AAZ90}) \label{lem:A(x) is finite iff Z(x) is finite}
%Let $M$ be a reduced (commutative) atomic monoid and let $x \in M$. Then $|\{a \in \mathcal{A}(M) : a \mid_M x \}| < \infty$ if and only if $|\mathsf{Z}(x)| < \infty$.
%\end{lemma}

\begin{lemma}(cf. \cite[Theorem~5.1]{AAZ90}) \label{lem:A(x) is finite iff Z(x) is finite}
	Let $M$ be a reduced and atomic commutative monoid, and take $x \in M$. Then the following statements are equivalent.
	\begin{enumerate}
		\item[(a)] $D_M(x)$ is a finite set.
		\smallskip
		
		\item[(b)] $A_M(x)$ is a finite set.
		\smallskip
		
		\item[(c)] $\mathsf{Z}_M(x)$ is a finite set.
	\end{enumerate}%$|\{a \in \mathcal{A}(M) : a \mid_M x \}| < \infty$ if and only if 
\end{lemma}

\begin{proof}
	(a) $\Rightarrow$ (b): It is clear.
	\smallskip
	
	(b) $\Rightarrow$ (c): Now suppose that $A_M(x)$ is a finite set, namely, $A_M(x) = \{a_1, \dots, a_n\}$ for some $n \in \nn$ and pairwise distinct atoms $a_1, \dots, a_n$ of $M$. Now consider the submonoid $N = \langle a_1, \dots, a_n \rangle$ of $M$. Because $M$ is a reduced monoid, $\mathcal{A}(M) \cap N \subseteq \mathcal{A}(N)$ and, therefore, $N$ is atomic with $\mathcal{A}(N) = A_M(x)$. It is then clear that $x \in N$ and $A_N(x) = A_M(x)$. Thus, $\mathsf{Z}_M(x) = \mathsf{Z}_N(x)$. It follows from~\cite[Theorem 3.1.4]{GH06} that $N$ is an FFM and, consequently,  $|\mathsf{Z}_M(x)| = |\mathsf{Z}_N(x)| < \infty$.
	\smallskip
	
	(c) $\Rightarrow$ (a): Finally, suppose that $\mathsf{Z}_M(x)$ is a finite set. Take $d \in M$ such that $d \mid_M x$, and then notice that after writing $x = dd'$ for some $d' \in M$ and factoring both $d$ and $d'$ in $M$, one obtains a factorization of $x$. Hence every factorization of $d$ in $M$ is a subfactorization of $x$. Since $\mathsf{Z}_M(x)$ is finite and each factorization of $x$ has only finitely many subfactorizations, $D_M(x)$ is also finite.
\end{proof}

%\begin{proof}
%Set $A_x = \{a \in \mathcal{A}(M) : a \mid_M x \}$. Since $M$ is reduced, $|\mathsf{Z}(x)| < \infty$ immediately implies that $|A_x| < \infty$. Conversely, suppose that $A_x$ is a finite set, say $A_x = \{a_1, \dots, a_n\}$, and consider the submonoid $N = \langle a_1, \dots, a_n \rangle$ of $M$. As $M$ is reduced, $\mathcal{A}(M) \cap N \subseteq \mathcal{A}(N)$, and so $N$ is atomic with $\mathcal{A}(N) = A_x$. It is then clear that $x \in N$ and $\{a \in \mathcal{A}(N) : a \mid_N x\} = A_x$. Thus, $\mathsf{Z}_M(x) = \mathsf{Z}_N(x)$. It follows from~\cite[Theorem 3.1.4]{GH06} that $N$ is an FFM and, consequently,  $|\mathsf{Z}_M(x)| = |\mathsf{Z}_N(x)| < \infty$.
%\end{proof}

\smallskip

The atomicity of $(S_r,+)$ was studied in~\cite{GG17}. We are ready now to explore the atomicity of the information semialgebra $S_r$. First, we introduce the following notation.

\smallskip

\noindent{\bf Notation.} For $r \in \qq^\bullet$, we define the \emph{support} of $r$ to be the set $\supp(r) = \{p \in \pp : p \mid \mathsf{n}(r) \text{ or } p \mid \mathsf{d}(r) \}$. In addition, for any subset $R$ of $\qq^\bullet$, we set $\supp(R) := \bigcup_{r \in R} \supp(r)$.

\smallskip

\begin{prop} \label{prop:atomicity of cyclic PIS}
Let $r \in \qq_{> 0}$ and consider the Puiseux information semialgebra $S_r$. The following statements hold.
\begin{enumerate}
	\item[1.] If $\mathsf{d}(r)=1$, then
	\begin{itemize}
		\item $(S_r,+)$ is atomic with $\mathcal{A}(S_r) = \{1\}$ and
		\smallskip
		
		\item $(S_r^\bullet, \cdot)$ is atomic with $\mathcal{A}_\times(S_r) = \pp$.
	\end{itemize}
\vspace{2pt}
	\item[2.] If $\mathsf{d}(r) > 1$ and $\mathsf{n}(r) = 1$, then
	\begin{itemize}
		\item $(S_r,+)$ is antimatter, and
		\smallskip
		
		\item $(S_r^\bullet, \cdot)$ is atomic.
	\end{itemize}
\vspace{2pt}
	\item[3.] If $\mathsf{d}(r) > 1$ and $\mathsf{n}(r) > 1$, then 
	\begin{itemize}
		\item $S_r$ is atomic with $\mathcal{A}(S_r) = \{r^n : n \in \nn_0\}$, and
		\smallskip
		
		\item $(S_r^\bullet, \cdot)$ is atomic provided that either $r > 1$ or $\mathsf{d}(r) \in \pp$.
	\end{itemize}
\end{enumerate}
\end{prop}

\begin{proof}
1. If $\mathsf{d}(r) = 1$, then it is clear that $S_r^\bullet = \nn$. As a result, $S$ is bi-atomic, $\mathcal{A}_+(S_r) = \{1\}$, and $\mathcal{A}_\times(S_r) = \pp$.
\smallskip
	
2. Suppose that $\mathsf{d}(r) > 1$ and $\mathsf{n}(r) = 1$. Since $r^n = \mathsf{d}(r) r^{n+1}$, the monoid $(S_r,+)$ is antimatter. To see that $(S_r^\bullet, \cdot)$ is atomic, first observe that $S_r^\bullet = \big\{ \frac{n}{\mathsf{d}(r)^m} : m,n \in \nn_0 \big\}$. %If $q \in \qq_{>0}$ satisfies that $\supp(q) \subseteq \supp(r)$, then
%\[
%	\supp(\mathsf{d}(q^{-1})) = \supp(\mathsf{n}(q)) \subseteq \supp(q) \subseteq \supp(\mathsf{d}(r))
%\]
%and, therefore, $q \in S_r^\times$. On the other hand, it is clear that $q \in S_r^\times$ implies that $\supp(q) \subseteq \supp(r)$.
From this observation, one can deduce that $S_r^\times = \{q \in \qq_{>0} : \supp(q) \subseteq \supp(r) \}$ and so the reduced monoid of $(S_r^\bullet, \cdot)$ is isomorphic to the multiplicative monoid $\{n \in \nn : \supp(n) \cap \supp(r) = \emptyset \}$, which is the (multiplicative) free commutative monoid on the set $\pp \setminus \supp(r)$. As a consequence, $(S_r^\bullet, \cdot)$ is atomic.
\smallskip
	
3. Suppose that $\mathsf{d}(r) > 1$ and $\mathsf{n}(r) > 1$. It was shown in~\cite[Theorem~6.2]{GG17} that $(S_r,+)$ is atomic and $\mathcal{A}_+(S_r) = \{r^n : n \in \nn_0\}$. We proceed to argue that $(S_r^\bullet, \cdot)$ is atomic in the specified cases.
\smallskip
	
\noindent {\it Case 1}: $r > 1$. Observe that for each $n \in \nn$, there are only finitely many elements of $S_r^\bullet$ that are less than $n$. Then $0$ cannot be a limit point of $(\log S_r^\bullet ) \setminus \{0\}$, and it follows from~\cite[Proposition~4.5]{fG19} that $(\log S_r^\bullet, +)$ is a BFM. Since $(S_r^\bullet, \cdot) \cong (\log S_r^\bullet,+)$, the former is a BFM and, therefore, an atomic monoid.
\smallskip
	
\noindent {\it Case 2}: $\mathsf{d}(r) \in \pp$. By appealing to the previous case, there is no loss in assuming that $r < 1$. Fix $x \in S_r^\bullet$. By~\cite[Lemma~3.1]{CGG19} there exists a unique factorization $z = \sum_{n \in \nn_0} c_n r^n \in \mathsf{Z}_{S_r}(x)$ such that $c_n < \mathsf{d}(r)$ for every $n \in \nn$. Take $y, y' \in S_r^\bullet$ such that $x = yy'$, and let $w = \sum_{n \in \nn_0} b_n r^n$ and $w' = \sum_{n \in \nn_0} b'_n r^n$ be factorizations in $\mathsf{Z}_{S_r}(y)$ and $\mathsf{Z}_{S_r}(y')$, respectively, satisfying $b_n, b'_n < \mathsf{d}(r)$ for every $n \in \nn$. Set $k := \max\{i \in \nn_0 : b_i \neq 0\}$ and $\ell := \max \{i \in \nn_0 : b'_i \neq 0 \}$. As $0 < b_k, b'_\ell < \mathsf{d}(r) \in \pp$, one can see that $\mathsf{d}(r) \nmid b_k b'_\ell$. Then $c_{k + \ell} \neq 0$ and so $k \le k + \ell \le m_0 = \max \{i \in \nn_0 : c_i \neq 0\}$. So each divisor of $x$ has a factorization of the form $\sum_{i=0}^{m_0} b_i r^i$, where $b_i < \mathsf{d}(r)$ for every $i \in \ldb 1,m_0 \rdb$ and $b_0 < x$. Thus, $x$ has only finitely many divisors in $(S_r^\bullet, \cdot)$. Since $S_r$ is reduced by Lemma~\ref{lem:reducibility of S_r}, it follows from Lemma~\ref{lem:A(x) is finite iff Z(x) is finite} that $(S_r^\bullet, \cdot)$ is an FFM. and, therefore, an atomic monoid.
\end{proof}

\smallskip

With notation as in Proposition~\ref{prop:atomicity of cyclic PIS}, we believe that $(S_r^\bullet, \cdot)$ is always atomic regardless of whether $\mathsf{d}(r)$ is or not a prime number. This would follow from Proposition~\ref{prop:atomicity of cyclic PIS} if the following conjecture holds. 

\smallskip

\begin{conj}
If $r \in \qq_{>0}$ is such that $\mathsf{n}(r) > 1$ and $\mathsf{d}(r) > 1$, then $(S_r^\bullet, \cdot)$ satisfies the ACCP.
\end{conj}

\medskip

%%%%%%%%%%%%%%%%%%%%%%%%%%%%%%%%%%%%%%%
\subsection{Conductive Puiseux Information Semialgebras} We conclude this section describing the atomicity of another class of Puiseux information semialgebras, those whose underlying Puiseux monoids have nonempty conductors.

%Recall that the \emph{difference group}\footnote{The difference group is called the quotient group when monoids are written multiplicatively.} $\gp(M)$ of a monoid $M$ is the abelian group (unique up to isomorphism) satisfying that any abelian group containing a homomorphic image of $M$ also contains a homomorphic image of $\gp(M)$. Clearly, the difference group of a Puiseux monoid $M$ (see~\cite[Proposition~3.5]{CGG20} for its description) can be taken to be an additive submonoid of $\qq$, namely, $\gp(M) = \{q-r : q,r \in M\}$.

%The \emph{conductor} of $M$ is defined to be
%\begin{equation} \label{eq:conductor}
%\mathfrak{c}(M) := \{ x \in \gp(M) \mid x + \widehat{M} \subseteq M \}.
%\end{equation}

%The \emph{rank} of a monoid $M$ is the rank of the $\zz$-module $\gp(M)$ or, equivalently, the dimension of the $\qq$-vector space $\qq \otimes_\zz \gp(M)$.

\begin{definition} \label{def:conductive PIS}
	For $r \in \rr_{\ge 0}$, we say that a Puiseux information semialgebra $Q$ is \emph{conducted} by $r$ provided that $Q$ is generated as a semiring by the set $\qq_{\ge r}$.
\end{definition}

We let $Q_r$ denote the Puiseux information semialgebra conducted by $r$. The terminology in Definition~\ref{def:conductive PIS} is motivated by the fact that for all $r \in \rr_{\ge 1}$ the semialgebra $Q_r$ is the smallest Puiseux information semialgebra whose underlying Puiseux monoid has conductor $\qq_{\ge r}$ (the conductor of a Pusieux monoid has been recently described in~\cite[Section~3]{GGT19}).

\smallskip

\begin{prop} \label{prop:atomicity of elementary conductive PIA}
For $r \in \rr_{\ge 0}$, let $Q_r$ be the Puiseux information semialgebra conducted by $r$. Then the following statements hold.
\begin{enumerate}
	\item[1.] If $r < 1$, then $Q_r = \qq_{\ge 0}$. In this case,
	\begin{itemize}
		\item $(Q_r,+)$ is antimatter and
		\smallskip
		
		\item $(Q^\bullet_r, \cdot)$ is atomic with $\mathcal{A}_\times(Q_r) = \emptyset$.
	\end{itemize}
\vspace{2pt}
		
	\item[2.] If $r = 1$, then $Q_r = \{0\} \cup \qq_{\ge 1}$. In this case,
	\begin{itemize}
		\item $(Q_r,+)$ is atomic with $\mathcal{A}_+(Q_r) = [1,2)$,and
		\smallskip
	
		\item $(Q^\bullet_r, \cdot)$ is antimatter.
	\end{itemize}
	\vspace{2pt}
		
	\item[3.] If $r > 1$, then $Q_r = \nn_0 \cup \qq_{\ge r}$. In this case,
	\begin{itemize}
		\item $(Q_r,+)$ is atomic with $\mathcal{A}_+(Q_r) = \big( \{1\} \cup ([r,r+1) \cap \qq) \big) \setminus \{ \lceil r \rceil \}$ and
		\smallskip
		
		\item $(Q^\bullet_r, \cdot)$ is atomic with $\mathcal{A}_\times(Q_r) = \big( \pp_{< r^2} \cup ([r,r^2) \cap \qq) \big) \setminus \pp \cdot (Q_r)_{> 1}$.
	\end{itemize}
\end{enumerate}
\end{prop}

\begin{proof}
1. Assume that $r < 1$. Then take $r_0 \in Q_r \cap (0,1)$ and take $q \in \qq_{> 0}$. Note that we can choose $n \in \nn$ large enough so that $q_0 := r_0^{-n}q  > 1$. Then $q_0 \in \qq_{\ge 1} \subset Q_r$, which implies that $q = q_0 r_0^n \in Q_r$. Therefore $Q_r = \qq_{\ge 0}$. The Puiseux monoid $\qq_{\ge 0}$ is clearly antimatter, and the multiplicative monoid $\qq_{> 0}$ is an atomic monoid with no atoms because it is an abelian group.
\smallskip
	
2. It follows immediately that $Q_r = \{0\} \cup \qq_{\ge 1}$ when $r = 1$. We have already argued in Example~\ref{ex:conductive semiring} that the Puiseux monoid $Q_1$ is atomic with $\mathcal{A}_+(Q_1) = [1,2)$ and that the multiplicative monoid $(Q_r^\bullet, \cdot)$ is antimatter.
\smallskip
	
3. Finally, assume that $r > 1$. Clearly, $Q_r = \nn_0 \cup \qq_{\ge r}$. Because $(Q_r,+)$ is a reduced monoid and $\qq_{\ge r+1} \subseteq 1 + Q_r^\bullet$, we see that $\mathcal{A}_+(Q_r) \subseteq Q^\bullet_r \cap \qq_{< r+1} =  \ldb 1, \lceil r \rceil \rdb \cup ([r, r+1) \cup \qq)$. Since $1 \in \mathcal{A}_+(Q_r)$ and $m \notin \mathcal{A}_+(Q_r)$ for any $m \in \ldb 2, \lceil r \rceil \rdb$, we find that
\[
	\mathcal{A}_+(Q_r) = \big( \{1\} \cup ([r,r+1) \cap \qq) \big) \setminus \{ \lceil r \rceil \}
\]
To find the set of multiplicative atoms, first observe that $(Q_r)_{\ge r^2} \subseteq (Q_r)_{> 1} \cdot (Q_r)_{> 1}$. This, along with the fact that $(Q_r^\bullet, \cdot)$ is reduced, guarantees that $\mathcal{A}_\times(Q_r) \subseteq Q_r^\bullet \cap \qq_{< r^2} = \ldb 1, \lceil r \rceil \rdb \cup ([r,r^2) \cap \qq)$. On the other hand, the only elements in $\ldb 1, \lceil r \rceil \rdb \cup ([r,r^2) \cap \qq)$ that are not multiplicative atoms of $Q_r$ are those that are properly divisible in $(Q^\bullet_r, \cdot)$ by some prime number. Thus,
\[
	\mathcal{A}_\times(Q_r) = \big( \pp_{< r^2} \cup ([r,r^2) \cap \qq) \big) \setminus \pp \cdot (Q_r)_{> 1},
\]\
which completes the proof.
\end{proof}

\smallskip

\begin{cor}
	Let $Q_r$ be the Puiseux information semialgebra conducted by $r \in \rr_{\ge 0}$. Then $Q_r$ is reduced if and only if $r \ge 1$, which happens precisely when $Q_r = \qq_{\ge 0}$.
\end{cor}

\bigskip

%%%%%%%%%%%%%%%%%%%%%%%%%%%%%%%%%%%%%%%%%%%%
\section{Upper Triangular Matrices over Information Semialgebras}
\label{sec:upper triangular matrices}

A study of the atomic structure and computation of some factorization-theoretic invariants for multiplicative monoids of the form $T_n(S)^\bullet$ (i.e., monoids of upper triangular matrices over an information semialgebra $S$) was initiated in~\cite{BS20}. Special emphasis was placed on the Puiseux information semialgebra $S = \nn_0$. The main purpose of this section is to use some important factorization-theoretic tools from the commutative setting to help understand the atomicity of the multiplicative monoid $T_n(S)^\bullet$ and its submonoid $U_n(S)$ consisting of all $n \times n$ unit triangular matrices over $S$.

We first recall some terminology pertaining to divisibility in noncommutative settings. Let $M$ be a monoid, and take $x,y \in M$. We say that $x$ \emph{divides} $y$ \emph{up to permutation} if we can write $x = a_1 \dots a_m$ and $y = b_1 \dots b_n$ for $a_1, \dots, a_m, b_1, \dots, b_n \in \mathcal{A}(M)$ such that $m \le n$ and such that there is an injection $\sigma \colon \ldb 1,m \rdb \to \ldb 1,n \rdb$ with $b_i = a_{\sigma(i)}$ for each $i \in \ldb 1,m \rdb$. In this case, we write $x \mid_p y$. Note that this notion of divisibility coincides with the usual notion of divisibility when $M$ is commutative. Following~\cite{BS15}, we call an element $a \in M \setminus U(M)$ \emph{almost prime-like} if for all $x,y \in M$ the relation $a \mid_p xy$ implies that either $a \mid_p x$ or $a \mid_p y$. If $M$ is atomic, then each almost prime-like element is an atom~\cite[Lemma~5.5]{BS15}. Clearly, the notion of an almost prime-like element resemblances that of a prime element when $M$ is commutative.

We would like to emphasize that for a Puiseux information semialgebra $S$ different from $\nn_0$, the atomic structure of $T_n(S)^\bullet$ may significantly differ from that of $T_n(\nn_0)^\bullet$. The following example sheds some light upon this observation.

\smallskip

\begin{example} \label{ex:many almost prime-like}
Fix $r \in \qq_{>0} \setminus \nn$ and consider the cyclic Puiseux information semialgebra $S_r$. It follows from \cite[Lemma~3.1]{CGG19} and \cite[Lemma~3.2]{CGG19} that $|\mathsf{Z}_{S_r}(r)| = 1$. Suppose that
\begin{equation} \label{eq:many almost prime-like in T_2(S)}
	r = \bigg( \sum_{i=0}^m c_i r^i \bigg) \bigg( \sum_{j=0}^n c'_j r^j \bigg)
\end{equation}
for some $m,n \in \nn_0$ and $c_1, \dots, c_m, c'_1, \dots, c'_n \in \nn_0$ such that $c_m c'_n \neq 0$. Since $|\mathsf{Z}_{S_r}(r)| = 1$, the factorization $z$ of $r$ that one obtains after multiplying the two sums in the right-hand side of~(\ref{eq:many almost prime-like in T_2(S)}) must be~$r$ itself. As $c_m c'_n r^{m+n}$ appears in $z$, we see that $m+n = 1$ and $c_m = c'_n = 1$. Hence $\{m,n\} = \{0,1\}$, which implies that $r \in \mathcal{A}_\times(S_r)$.
	
We now show that each element of the form $\begin{psmallmatrix} 1 & r^k \\ 0 & 1\end{psmallmatrix}$ with $k \in \nn$ fails to be almost prime-like in $T_2(S_r)^\bullet$. Since $r \in \mathcal{A}_\times(S_r)$ and $\mathcal{A}_+(S_r) = \{r^n : n \in \nn_0 \}$, it follows from~\cite[Theorem~2.1]{BS20} that the matrices $\begin{psmallmatrix} r & 0 \\ 0 & 1 \end{psmallmatrix}$ and $\begin{psmallmatrix} 1 & r^{k-1} \\ 0 & 1 \end{psmallmatrix}$ are atoms of $T_2(S_r)^\bullet$ for every $k \in \nn$. For each $k \in \nn$ verifying that $\begin{psmallmatrix} 1 & r^k \\ 0 & 1\end{psmallmatrix}$ is not almost prime-like in $T_2(S_r)^\bullet$ amounts to observing that
\[
	\begin{pmatrix} 1 & r^k \\ 0 & 1 \end{pmatrix}\begin{pmatrix}r & 0 \\ 0 & 1\end{pmatrix} = \begin{pmatrix} r & 0 \\ 0 & 1 \end{pmatrix}\begin{pmatrix}1 & r^{k-1} \\ 0 & 1\end{pmatrix}.
\]
\end{example}

\smallskip

In light of Example~\ref{ex:many almost prime-like} and Proposition~\ref{prop:algebraic observation of T_n(S)} below, the matrices $I_n + r^k E_{ij}$ fail to be almost prime-like elements in $T_n(S_r)^\bullet$ for every $k \in \nn_0$ and every $i,j \in \ldb 1,n \rdb$ with $i < j$. This makes it clear that the atomic structure of $T_n(S_r)^\bullet$ is quite different from that of $T_n(\nn_0)^\bullet$ since, in the latter case (see \cite[Remark~2.11]{BS20}), the set of almost prime-like atoms consists of the matrices $I_n + (p-1)E_{ii}$ with $i \in \ldb 1,n \rdb$ and the matrices $p \in \pp$ and $I_n + E_{ij}$ with $i \in \ldb 1,n-1 \rdb$ and $j= i+1$.

\medskip

\subsection{Divisibility and Atomicity} Our next immediate goal is to collect a few preliminary results about the semiring $T_n(S)$ and the multiplicative monoid $T_n(S)^\bullet$, whenever $S$ is a reduced information semialgebra. We begin by recalling some relevant terminology.

Let $M$ be a monoid. If $x,y \in M$, then $x$ is a \emph{left divisor} (resp., \emph{right divisor}) of $y$ in $M$ if $y \in xM$ (resp., $y \in Mx$), and $x$ \emph{divides} $y$ in $M$ if $x$ is either a left or a right divisor of $y$ in $M$. A submonoid $N$ of $M$ is called \emph{divisor-closed} provided that for all $y \in N$ the fact that $x \in M$ divides $y$ in $M$ enforces $x \in N$. Divisor-closed submonoids are perhaps the most relevant submonoids in terms of atomicity and factorizations as they inherit the atomic properties of the monoids containing them. In the next proposition we identify some of the divisor-closed submonoids of $T_n(S)^\bullet$. Such submonoids play a crucial role in the rest of the paper.

\smallskip

\begin{prop} \label{prop:algebraic observation of T_n(S)}
Let $S$ be a reduced information semialgebra and take $n \in \mathbb N$. The following statements hold.
\begin{enumerate}
	\item[1.] $U_n(S)$ is a divisor-closed submonoid of $T_n(S)^\bullet$.
	\smallskip
	
	\item[2.] $\{I_n + (s-1)E_{ii} : s \in S^\bullet \} \cong (S^\bullet, \cdot)$ is a divisor-closed submonoid of $T_n(S)^\bullet$ for every $i \in \ldb 1, n \rdb$.
	\smallskip
	
	\item[3.] \label{item lemma:(S,plus) is a divisor-closed submonoid of T_n(S)} If $n \ge 2$, then $\{ I_n + sE_{ij} : s \in S \} \cong (S,+)$ is a divisor-closed submonoid of both monoids $U_n(S)$ and $T_n(S)^\bullet$ for all $i,j \in \ldb 1,n \rdb$ with $i < j$.
	\end{enumerate}
\end{prop}

\begin{proof}
1. Suppose that $U \in U_n(S)$ and write $U = BB'$ for some $B, B' \in T_n(S)^\bullet$. Because $S$ is reduced both matrices $B$ and $B'$ must have unit diagonal entries and so they belong to $U_n(S)$. As a result, $U_n(S)$ is a divisor-closed submonoid of $T_n(S)^\bullet$.
\smallskip
	
2. Fix $i \in \ldb 1,n \rdb$, and set $S_i = \{I_n + (s-1)E_{ii} : s \in S^\bullet \}$. Following the notation in~\cite{BS20}, we set $\Sigma(A) := \sum_{1 \le i < j \le n} \max \mathsf{L}_{S}(A_{ij})$ for each $A \in T_n(S)^\bullet$. It is clear that $\Sigma(A) = 0$ if and only if $A$ is a diagonal matrix. To verify that $S_i$ is a divisor-closed submonoid of $T_n(S)^\bullet$, take $A \in S_i$ and $B, B' \in T_n(S)^\bullet$ such that $A = BB'$. Since $\Sigma(A) = \Sigma(BB') \ge \Sigma(B) + \Sigma(B')$ by \cite[Lemma~3.4]{BS20}, both $B$ and $B'$ are diagonal matrices. Since $S$ is reduced and $A_{jj} = B_{jj}B'_{jj}$ for every $j \in \ldb 1,n \rdb$, we see that $B, B' \in S_i$. Thus, $S_i$ is a divisor-closed submonoid of $T_n(S)^\bullet$. It is clear that the assignment $s \mapsto I_n + (s-1)E_{ii}$ determines an isomorphism between $(S^\bullet, \cdot)$ and $S_i$.
\smallskip
	
3. Fix a pair $i,j \in \ldb 1,n \rdb$ with $i < j$, and set $S_{ij} := \{ I_n + sE_{ij} : s \in S \}$. Observe that for $x,y \in S$, the equality $(I_n + xE_{ij})(I_n + yE_{ij}) = I_n + (x+y)E_{ij}$ holds. Therefore the natural map $s \mapsto I_n + sE_{ij}$ gives the desired isomorphism. Note that as $S$ is a reduced information semialgebra, if $I_n + sE_{ij} = BB'$ for some $s \in S$ and $B,B' \in T_n(S)^\bullet$, then every diagonal entry of $B$ and $B'$ must be~$1$. Moreover, if $B_{rs} \neq 0$ for some $r < s$, then the entry of $BB'$ in position $(r,s)$ would be different from zero, a contradiction. Similarly,~$B'$ can only have a single nonzero entry off the diagonal, namely, in position $(i,j)$. Hence $S_{ij}$ is a divisor-closed submonoid of $T_n(S)^\bullet$, which implies that $S_{ij}$ is also a divisor-closed submonoid of $U_n(S)$.
\end{proof}

\smallskip

If $S$ is a reduced information semialgebra, then the divisor-closed submonoids of $T_n(S)^\bullet$ found in Proposition~\ref{prop:algebraic observation of T_n(S)} give us enough information to study the most relevant atomic properties of both $U_n(S)$ and $T_n(S)^\bullet$. The set of atoms of $T_n(S)^\bullet$ was fully described in~\cite[Theorem~2.1]{BS20} in terms of the atoms of $S$ as follows: $\mathcal{A}(T_n(S)^\bullet) = A_+ \cup A_\times$, where
\begin{align*} \label{eq:atoms of T_n(S)}
	A_+ &= \big\{ I_n + aE_{ij} : 1 \le i < j \le n \text{ and } a \in \mathcal{A}_+(S) \big\} \text{ and }\\
	A_\times &= \big\{ I_n + (a-1)E_{ii} : 1 \le i \le n \text{ and } a \in \mathcal{A}_\times(S^\bullet) \big\}.
\end{align*}
We call the elements of $A_+$ \emph{atoms of additive type} and the elements of $A_\times$ \emph{atoms of multiplicative type}. It was also shown in \cite[Theorem~2.1]{BS20} that $T_n(S)^\bullet$ is atomic provided that $S$ is bi-atomic. This result can be slightly extended as follows.

\smallskip

\begin{theorem} \label{thm:atomicity}
Let $S$ be a reduced information semialgebra.
\begin{enumerate}
	\item[1.] The following statements are equivalent.
	\begin{enumerate}
		\item[(a)] $U_n(S)$ is atomic for every $n \in \nn_{\ge 2}$.
		\smallskip
	
		\item[(b)] $U_n(S)$ is atomic for some $n \in \nn_{\ge 2}$.
		\smallskip
		
		\item[(c)] $(S,+)$ is atomic.
	\end{enumerate}

\smallskip
		
	\item[2.] The following statements are equivalent.
	\begin{enumerate}
		\item[(a)] $T_n(S)^\bullet$ is atomic for every $n \in \nn_{\ge 2}$.
		\smallskip
		
		\item[(b)] $T_n(S)^\bullet$ is atomic for some $n \in \nn_{\ge 2}$.
		\smallskip
		
		\item[(c)] $S$ is bi-atomic.
	\end{enumerate}
\end{enumerate}
\end{theorem}

\begin{proof}
1. (a) $\Rightarrow$ (b): This is clear.
\smallskip
	
(b) $\Rightarrow$ (c): By part~3 of Proposition~\ref{prop:algebraic observation of T_n(S)}, it follows that $\{ I_n + sE_{12} : s \in S \}$ is a divisor-closed submonoid of $U_n(S)$ isomorphic to $(S,+)$. Thus, $(S,+)$ is atomic when $U_n(S)$ is atomic.
\smallskip
	
(c) $\Rightarrow$ (a): Fix $n \in \nn_{\ge 2}$. By part~1 of Proposition~\ref{prop:algebraic observation of T_n(S)}, the monoid $U_n(S)$ is a divisor-closed submonoid of $T_n(S)^\bullet$. Then clearly $\mathcal{A}(U_n(S)) = U_n(S) \cap \mathcal{A}(T_n(S)^\bullet)$, which is precisely the set $A_+$ of atoms of additive type of $T_n(S)^\bullet$. As $(S,+)$ is atomic, $U_n(S)$ can be generated by $A_+$. Thus, $U_n(S)$ is atomic.
\smallskip
	
2. (a) $\Rightarrow$ (b): This is clear.
\smallskip
	
(b) $\Rightarrow$ (c): By parts~2 and~3 of Proposition~\ref{prop:algebraic observation of T_n(S)}, it follows that $\{ I_n + (s-1)E_{11} : s \in S^\bullet \}$ and $\{ I_n + sE_{12} : s \in S \}$ are divisor-closed submonoids of $T_n(S)^\bullet$ isomorphic to $(S^\bullet, \cdot)$ and $(S,+)$, respectively. Hence $S$ is bi-atomic provided that $T_n(S)^\bullet$ is atomic.
\smallskip
	
(c) $\Rightarrow$ (a): See~\cite[Theorem~2.1]{BS20}.
\end{proof}

\smallskip

At the opposite end of the spectrum from atomic monoids is that of monoids that are antimatter. Since the set of atoms of $T_n(S)^\bullet$ is parameterized by the sets of atoms of $(S,+)$ and $(S,\cdot)$, the next observation immediately follows.

\smallskip

\begin{remark}
 For every $n \in \nn_{\ge 2}$ and each information semialgebra $S$, the noncommutative monoid $T_n(S)^\bullet$ is antimatter if and only if both monoids $(S,+)$ and $(S, \cdot)$ are antimatter.
\end{remark}

\medskip

%%%%%%%%%%%%%%%
\subsection{The ACCP} Recall that a monoid that satisfies both the ACCP on left ideals and the ACCP on right ideals is atomic~\cite[Proposition~3.1]{dS13}. The converse is not true, even in the context of commutative monoids. For instance, the Puiseux monoid $\big\langle \frac{1}{2^p p} : p \in \pp_{\ge 2} \big\rangle$ is an atomic monoid that does not satisfy the ACCP. We proceed to study how, for each $n \in \nn_{\ge 2}$, the ACCP transfers between~$S$ and the monoids $U_n(S)$ and $T_n(S)^\bullet$. We first show that the ACCP can be transferred back and forth between $(S,+)$ and $U_n(S)$.

\smallskip

\begin{theorem} \label{thm:ACCP characterization for U_n(S)}
Let $S$ be a reduced information semialgebra. The following statements are equivalent.
\begin{enumerate}
	\item[(a)] $(S,+)$ is an ACCP monoid.
	\smallskip

	\item[(b)] $U_n(S)$ satisfies the right ACCP for every $n \in \nn_{\ge 2}$.
	\smallskip

	\item[(c)] $U_n(S)$ satisfies the left ACCP for every $n \in \nn_{\ge 2}$.
\end{enumerate}
\end{theorem}

\begin{proof}
\text{[(b) or (c)]} $\Rightarrow$ (a): It follows from part~3 of Proposition~\ref{prop:algebraic observation of T_n(S)} that $(S,+)$ is isomorphic to a divisor-closed submonoid of $U_n(S)$. As a result, if $U_n(S)$ satisfies either the left ACCP or the right ACCP, then $(S,+)$ must satisfy the ACCP.
\smallskip
	
(a) $\Rightarrow$ [(b) and (c)]: Suppose that $(S,+)$ satisfies the ACCP, and let $(A_k \, U_n(S))_{k \in \nn_0}$ be an ascending chain of principal right ideals of $U_n(S)$. Then there exists a sequence $(B_k)_{k \in \nn}$ of matrices in $U_n(S)$ such that $A_k = A_{k+1}B_{k+1}$ for every $k \in \nn_0$. We will find a finite increasing sequence $(k_t)_{t \in \ldb 1,n-1 \rdb}$ of positive integers such that $[B_k]_{j(j+t)} = 0$ for each $j \in \ldb 1, n-t \rdb$ and $k \ge k_t$. In particular, we will obtain that $B_k = I_n$ for every $k \ge k_{n-1}$.

We proceed inductively. Observe that for each $j \in \ldb 1,n-1 \rdb$ and $k \in \nn_0$,
\[
	[A_k]_{j(j+1)} = [B_{k+1}]_{j(j+1)} + [A_{k+1}]_{j(j+1)} = [B_{k+1}]_{j(j+1)} + [A_{k+1}]_{j(j+1)}.
\]
Then, for each $j \in \ldb 1,n-1 \rdb$, the sequence $( [A_k]_{j(j+1)} + S)_{k \in \nn_0}$ is an ascending chain of principal ideals of $(S,+)$ and must therefore stabilize. As $(S,+)$ is a reduced monoid, there exists $k_1 \in \nn$ such that $[B_k]_{j(j+1)} = 0$ for every $k \ge k_1$ and $j \in \ldb 1, n-1 \rdb$. Thus, the matrix $B_k$ has only $0$s on its superdiagonal when $k \ge k_1$.
	
For the inductive step, suppose that for $t \in \ldb 1,n-2 \rdb$ we have constructed a sequence $(k_i)_{i \in \ldb 1,t \rdb}$ satisfying the desired properties. In particular, for each $k \ge k_t$, $l \in \ldb 1,t \rdb$, and $j \in \ldb 1, n-l \rdb$, the equality $[B_k]_{j(j+l)} = 0$ holds and, as a result,
\[
	[A_k]_{j(j+t+1)} = \sum_{i=1}^n [A_{k+1}]_{ji}[B_{k+1}]_{i(j+t+1)} = [B_{k+1}]_{j(j+t+1)} + [A_{k+1}]_{j(j+t+1)}.
\]
Then, for each $j \in \ldb 1,n-t-1 \rdb$, the sequence $\left([A_k]_{j(j+t+1)} + S \right)_{k \ge k_t}$ is an ascending chain of principal ideals of $(S,+)$, which must stabilize. Thus, there exists $k_{t+1} \ge k_t$ such that $[B_{k+1}]_{j(j+t+1)} = 0$ for each $j \in \ldb 1,n-t-1 \rdb$ and $k \ge k_{t+1}$. As a result, $B_k = I_n$ for every $k \ge k_n$. Hence the ascending chain $(A_k U_n(S))_{k \in \nn_0}$ of principal right ideals of $U_n(S)$ must stabilize and (b) holds.
\smallskip
	
The proof that every ascending chain of principal left ideals of $U_n(S)$ stabilizes is almost identical, which concludes the proof.
\end{proof}

\smallskip

Without appealing to~\cite[Proposition~3.1]{dS13}, Theorem~\ref{thm:ACCP characterization for U_n(S)} allows us to deduce that if a monoid of unit triangular matrices over a reduced information semialgebra satisfies the ACCP, then it is atomic. We highlight this implication in the following corollary.

\smallskip

\begin{cor} \label{cor:ACCP implies atomicity for U_n(S)}
Let $S$ be a reduced information semialgebra. If $U_n(S)$ satisfies the ACCP, then $U_n(S)$ is atomic.
\end{cor}

\smallskip

Putting together Theorem~\ref{thm:ACCP characterization for U_n(S)} and Corollary~\ref{cor:ACCP implies atomicity for U_n(S)}, we obtain the next diagram of implications for each reduced information semialgebra $S$ and for every $n \in \nn_{\ge 2}$.
\begin{equation} \label{diag:from ACCP to atomicity for U_n(S)}
	\begin{tikzcd}
		&  (S,+) \text{ is \textbf{ACCP}} \arrow[r, Rightarrow] \arrow[d, Leftrightarrow] & (S,+) \text{ is \textbf{atomic}} \arrow[d, Leftrightarrow] \\ & U_n(S) \text{ is \textbf{ACCP}} \arrow[r, Rightarrow] & U_n(S) \text{ is \textbf{atomic}}
	\end{tikzcd}
\end{equation}

\smallskip

We now consider the monoid $T_n(S)^\bullet$.

\smallskip

\begin{theorem} \label{thm:ACCP characterization}
Let $S$ be a reduced information semialgebra. The following statements are equivalent.
\begin{enumerate}
	\item[(a)] $S$ is a bi-ACCP monoid.
	\smallskip

	\item[(b)] $T_n(S)^\bullet$ satisfies the right ACCP for every $n \in \nn$.
	\smallskip

	\item[(c)] $T_n(S)^\bullet$ satisfies the left ACCP for every $n \in \nn$.
\end{enumerate}
\end{theorem}

\begin{proof}
\text{[(b) or (c)]} $\Rightarrow$ (a): It follows from Proposition~\ref{prop:algebraic observation of T_n(S)} that both $(S,+)$ and $(S^\bullet, \cdot)$ are isomorphic to divisor-closed submonoids of $T_n(S)^\bullet$. Thus, if $T_n(S)^\bullet$ satisfies the right ACCP (or the left ACCP), then both monoids $(S,+)$ and $(S^\bullet, \cdot)$ satisfy the ACCP.
\smallskip
	
(a) $\Rightarrow$ [(b) and (c)]: Suppose now that both $(S,+)$ and $(S^\bullet, \cdot)$ satisfy the ACCP. Let $(A_k \, T_n(S)^\bullet)_{k \in \nn_0}$ be an ascending chain of principal right ideals of $T_n(S)^\bullet$. Then there exists a sequence $(B_k)_{k \in \nn}$ of matrices in $T_n(S)^\bullet$ such that $A_k = A_{k+1}B_{k+1}$ for every $k \in \nn_0$. Then, for each $j \in \ldb 1,n \rdb$, one obtains that $[A_k]_{jj} = [A_{k+1}]_{jj} [B_{k+1}]_{jj}$ for every $k \in \nn_0$. This gives rise to an ascending chain $( [A_k]_{jj} S^\bullet )_{k \in \mathbb N}$ of principal ideal of $(S^\bullet, \cdot)$. By hypothesis, these chains must stabilize. As $S$ is reduced, there exists $k_0 \in \nn$ such that for each $j \in \ldb 1,n \rdb$, we see that $[B_k]_{jj} = 1$ for every $k \ge k_0$. Thus, $B_k$ is a unit triangular matrix for every $k \ge k_0$.
\smallskip
	
The second part of the proof is similar to the proof of the implication (a) $\Rightarrow$ [(b) and (c)] of Theorem~\ref{thm:ACCP characterization for U_n(S)} and so we will only provide a brief sketch. For each $j \in \ldb 1,n-1 \rdb$ and $k \in \nn_0$,
\[
	[A_k]_{j(j+1)} = [A_{k+1}]_{jj}[B_{k+1}]_{j(j+1)} + [A_{k+1}]_{j(j+1)}[B_{k+1}]_{(j+1)(j+1)} = [A_{k+1}]_{jj}[B_{k+1}]_{j(j+1)} + [A_{k+1}]_{j(j+1)}.
\]
Then, for each $j \in \ldb 1,n-1 \rdb$, the sequence $( [A_k]_{j(j+1)} + S)_{k \in \nn_0}$ is an ascending chain of principal ideals of $(S,+)$. By hypothesis, this chain must stabilize. Thus, there exists $k_1 \ge k_0$ such that $[B_{k+1}]_{j(j+1)} = 0$ for each $j \in \ldb 1,n-1 \rdb$ and $k \ge k_1$. Suppose that for $t \in \ldb 1,n-2 \rdb$ there exists $k_t \in \nn$ with $k_t \ge k_0$ such that for each $k \ge k_t$, $l \in \ldb 1,t \rdb$, and $j \in \ldb 1, n-l \rdb$, the equality $[B_k]_{j(j+l)} = 0$ holds. Then
\[
	[A_k]_{j(j+t+1)} = \sum_{i=1}^n [A_{k+1}]_{ji}[B_{k+1}]_{i(j+t+1)} = [A_{k+1}]_{jj}[B_{k+1}]_{j(j+t+1)} + [A_{k+1}]_{j(j+t+1)}.
\]
Now, for each $j \in \ldb 1,n-t-1 \rdb$, the sequence $\left([A_k]_{j(j+t+1)} + S \right)_{k \ge k_t}$ is an ascending chain of principal ideals of $(S,+)$ and must stabilize. Thus, there exists $k_{t+1} \ge k_t$ such that $[B_{k+1}]_{j(j+t+1)} = 0$ for each $j \in \ldb 1,n-t-1 \rdb$ and $k \ge k_{t+1}$. Inductively, we will obtain $k_n \in \nn$ with $k_n \ge k_0$ such that $B_k = I_n$ for every $k \ge k_n$. As a consequence, the chain of ideals $(A_k T_n(S)^\bullet)_{k \in \nn_0}$ stabilizes, and so $T_n(S)^\bullet$ satisfies the right ACCP. In a similar way one can show that $T_n(S)^\bullet$ satisfies the left ACCP.
\end{proof}

\smallskip

We immediately obtain the following corollary.

\smallskip

\begin{cor} \label{cor:ACCP implies atomicity}
Let $S$ be a reduced information semialgebra. If $T_n(S)^\bullet$ satisfies the ACCP, then $T_n(S)^\bullet$ is atomic.
\end{cor}

\smallskip

Putting together Theorem~\ref{thm:atomicity}, Theorem~\ref{thm:ACCP characterization}, and Corollary~\ref{cor:ACCP implies atomicity}, we obtain the following diagram for each reduced information semialgebra $S$ and for every $n \in \nn_{\ge 2}$.
\begin{equation} \label{diag:from ACCP to atomicity}
	\begin{tikzcd}
		&  S \text{ is \textbf{bi-ACCP}} \arrow[r, Rightarrow] \arrow[d, Leftrightarrow] & S \text{ is \textbf{bi-atomic}} \arrow[d, Leftrightarrow] \\
		& T_n(S)^\bullet \text{ is \textbf{ACCP}} \arrow[r, Rightarrow] & T_n(S)^\bullet \text{ is \textbf{atomic}}
	\end{tikzcd}
\end{equation}

\smallskip

We conclude this subsection by constructing an information semialgebra $S$ such that for every $n \in \nn$ the monoids $U_n(S)$ and $T_n(S)^\bullet$ are atomic but do not satisfy the ACCP. This illustrates that the horizontal implications in Diagram~(\ref{diag:from ACCP to atomicity}) are not, in general, an equivalence.

\smallskip

\begin{example}
Let $r \in (0,1) \cap \qq$ such that $\mathsf{n}(r) \neq 1$ and $\mathsf{d}(r) \in \pp$ and consider the cyclic Puiseux information semialgebra~$S_r$. It was shown in Proposition~\ref{prop:atomicity of cyclic PIS} that $S_r$ is a bi-atomic monoid. Thus, it follows from Theorem~\ref{thm:atomicity} that for every $n \in \nn_{\ge 2}$ the monoids $U_n(S_r)$ and $T_n(S_r)^\bullet$ are atomic. Because of Theorem~\ref{thm:ACCP characterization for U_n(S)} and Theorem~\ref{thm:ACCP characterization}, showing that none of the monoids $U_n(S_r)$ and $T_n(S_r)^\bullet$ (for every $n \in \nn_2$) satisfies the ACCP amounts to arguing that $S_r$ is not a bi-ACCP monoid. To do so we shall verify that $(S_r,+)$ does not satisfy the ACCP. Consider the sequence of principal ideals $(\mathsf{n}(r)r^n + S_r)_{n \in \nn_0}$. For each $n \in \nn_0$, notice that
\[
	\mathsf{n}(r)r^n = \mathsf{d}(r)r^{n+1} = (\mathsf{d}(r) - \mathsf{n}(r))r^{n+1} + \mathsf{n}(r)r^{n+1},
\]
which implies that $\mathsf{n}(r)r^n + S_r \subseteq \mathsf{n}(r)r^{n+1} + S_r$. Therefore $(\mathsf{n}(r)r^n + S_r)_{n \in \nn_0}$ is an ascending chain of principal ideals. Since the sequence $(\mathsf{n}(r)r^n)_{n \in \nn_0}$ strictly decreases to zero, the chain of principal ideals $(\mathsf{n}(r)r^n + S_r)_{n \in \nn_0}$ does not stabilize. As a consequence, $S_r$ fails to satisfy the ACCP.
\end{example}

\medskip

%%%%%%%%%%%%%%%%%%%%%%%%%%%%%%%%
\subsection{The Bounded Factorization Property} In this subsection we characterize when the monoids $U_n(S)$ and $T_n(S)^\bullet$ are BFMs in terms of the additive and multiplicative structures of $S$.

\smallskip

\begin{theorem} \label{thm:BFM characterizations for U_n(S)}
Let $S$ be a reduced information semialgebra. The following statements are equivalent.
\begin{enumerate}
	\item[(a)] $U_n(S)$ is a BFM for every $n \in \nn_{\ge 2}$.
	\smallskip

	\item[(b)] $U_n(S)$ is a BFM for some $n \in \nn_{\ge 2}$.
	\smallskip

	\item[(c)] $(S,+)$ is a BFM.
\end{enumerate}
\end{theorem}

\begin{proof}
(a) $\Rightarrow$ (b): This is obvious.
\smallskip
	
(b) $\Rightarrow$ (c): Take $n \in \nn_{\ge 2}$ such that $U_n(S)$ is a BFM. Since $n \ge 2$, we can use part~3 of Proposition~\ref{prop:algebraic observation of T_n(S)} to identify $(S,+)$ with a divisor-closed submonoid of $U_n(S)$. Since $(S,+)$ is a divisor-closed submonoid of $U_n(S)$, the equality $\mathcal{A}(S) = \mathcal{A}(U_n(S)) \cap S$ holds and so $\mathsf{Z}_S^*(x) = \mathsf{Z}_{U_n(S)}^*(x)$ for all $x \in S$. The fact that $U_n(S)$ is a BFM now implies that $(S,+)$ is a BFM as well.
\smallskip
	
(c) $\Rightarrow$ (a): Fix $n \in \nn_{\ge 2}$ and $A \in U_n(S)$. Since $S$ is a BFM, $\max \mathsf{L}_{S}(A_{ij}) < \infty$ for all $i,j \in \ldb 1,n \rdb$ with $i < j$. It follows from~\cite[Lemma~3.4]{BS20} that $\max \mathsf{L}_{U_n(S)}(A) = \sum_{1 \le i < j \le n} \max \mathsf{L}_{S}(A_{ij})$, and so $|\mathsf{L}_{U_n(S)}(A)| < \infty$. As a consequence, each matrix in $U_n(S)$ has a bounded set of lengths, which completes the proof.
\end{proof}

\smallskip

Since each commutative BFM satisfies the ACCP, we obtain the following corollary.

\smallskip

\begin{cor} \label{cor:BFM implies ACCP for U_n(S)}
Let $S$ be a reduced information semialgebra. If $U_n(S)$ is a BFM, then $U_n(S)$ is an ACCP monoid.
\end{cor}

\smallskip

Using Theorem~\ref{thm:BFM characterizations for U_n(S)} and Corollary~\ref{cor:BFM implies ACCP for U_n(S)} one can extend Diagram~\eqref{diag:from ACCP to atomicity for U_n(S)}: for each reduced information semialgebra $S$ and for every $n \in \nn_{\ge 2}$, all implications in the following diagram hold.
\begin{equation} \label{diag:atomic diagram with BFM included for U_n(S)}
	\begin{tikzcd}
		(S,+) \text{ is \textbf{BFM}} \arrow[r, Rightarrow] \arrow[d, Leftrightarrow] &  (S,+) \text{ is \textbf{ACCP}} \arrow[r, Rightarrow] \arrow[d, Leftrightarrow] & (S,+) \text{ is \textbf{atomic}} \arrow[d, Leftrightarrow] \\
		U_n(S) \text{ is a \textbf{BFM}} \arrow[r, Rightarrow] & U_n(S) \text{ is \textbf{ACCP}} \arrow[r, Rightarrow] & U_n(S) \text{ is \textbf{atomic}}
	\end{tikzcd}
\end{equation}

\smallskip

We are now able to determine precisely when $T_n(S)^\bullet$ is a BFM.

\smallskip

\begin{theorem} \label{thm:BFM characterizations}
Let $S$ be a reduced information semialgebra. Then the following statements are equivalent.
\begin{enumerate}
	\item[(a)] $T_n(S)^\bullet$ is a BFM for every $n \in \nn_{\ge 2}$.
	\smallskip

	\item[(b)] $T_n(S)^\bullet$ is a BFM for some $n \in \nn_{\ge 2}$.
	\smallskip
	
	\item[(c)] $S$ is a bi-BFM.
\end{enumerate}
\end{theorem}

\begin{proof}
(a) $\Rightarrow$ (b): This is clear.
\smallskip
	
(b) $\Rightarrow$ (c): Take $n \in \nn_{\ge 2}$ such that $T_n(S)^\bullet$ is a BFM. Since $n \ge 2$, part~2 and part~3 of Proposition~\ref{prop:algebraic observation of T_n(S)} guarantee that $T_n(S)^\bullet$ has divisor-closed submonoids isomorphic to $(S,+)$ and $(S^\bullet, \cdot)$. After using an argument similar to that used in the proof of the corresponding implication of Theorem~\ref{thm:BFM characterizations for U_n(S)}, one obtains that both $(S,+)$ and $(S^\bullet, \cdot)$ are BFMs. Hence $S$ is a bi-BFM.
\smallskip
	
(c) $\Rightarrow$ (a): Fix $n \in \nn_{\ge 2}$, and take $A \in T_n(S)^\bullet$. Since $(S, +)$ is a BFM, $A_{ij} \in S$ has bounded set of lengths in $(S,+)$ for all $i,j \in \ldb 1,n \rdb$ with $i < j$. Also, since $(S^\bullet, \cdot)$ is a BFM, $\det A \in S^\bullet$ must have bounded set of lengths in $(S^\bullet, \cdot)$. It now follows from~\cite[Theorem~3.6.1]{BS20} that the set of lengths $\mathsf{L}_{T_n(S)^\bullet}(A)$ is bounded above by $\sum_{1 \le i < j \le n} \max \mathsf{L}_{S}(A_{ij}) + \max \mathsf{L}_{S^\bullet}(\det A)$. Therefore the monoid $T_n(S)^\bullet$ is a BFM.
\end{proof}

\smallskip

As an immediate consequence we have the following corollary.

\smallskip

\begin{cor} \label{cor:BFM implies ACCP}
Let $S$ be a reduced information semialgebra. If $T_n(S)^\bullet$ is a BFM, then $T_n(S)^\bullet$ is an ACCP monoid.
\end{cor}

\smallskip

We can now apply Theorem~\ref{thm:BFM characterizations} and Corollary~\ref{cor:BFM implies ACCP} to extend Diagram~\eqref{diag:from ACCP to atomicity}. For each information semialgebra $S$ and for every $n \in \nn_{\ge 2}$, each implication in the following diagram holds.
\begin{equation} \label{diag:atomic diagram with BFM included}
	\begin{tikzcd}
		S \text{ is \textbf{bi-BFM}} \arrow[r, Rightarrow] \arrow[d, Leftrightarrow] &  S \text{ is \textbf{bi-ACCP}} \arrow[r, Rightarrow] \arrow[d, Leftrightarrow] & S \text{ is \textbf{bi-atomic}} \arrow[d, Leftrightarrow] \\
		T_n(S)^\bullet \text{ is a \textbf{BFM}} \arrow[r, Rightarrow] & T_n(S)^\bullet \text{ is \textbf{ACCP}} \arrow[r, Rightarrow] & T_n(S)^\bullet \text{ is \textbf{atomic}}
	\end{tikzcd}
\end{equation}

\smallskip

We emphasize that the leftmost horizontal implications in Diagram~\eqref{diag:atomic diagram with BFM included} are not, in general, reversible. Before we illustrate this observation in Example~\ref{example:LW}, we recall the Lindemann-Weierstrass Theorem, an important tool in transcendental number theory.

\smallskip

\begin{theorem} \cite[Chapter~1]{aB90}
If $\alpha_1, \dots, \alpha_n$ are distinct algebraic numbers, then the set $\{e^{\alpha_1}, \dots, e^{\alpha_n} \}$ is linearly independent over the algebraic numbers.
\end{theorem}

\smallskip

\begin{example}\label{example:LW}
Set $M = \langle 1/p : p \in \pp \rangle$ and consider the submonoid $S = \langle e^q : q \in M \rangle$ of $(\rr_{\ge 0},+)$. By the Lindemann-Weierstrass Theorem, $(S,+)$ is the free commutative monoid on the set $\{e^q : q \in M\}$. In addition, it is clear that $S$ is closed under multiplication. Therefore $S$ is an information semialgebra. Since $1 = \min \, S^\bullet$, the monoid $(S^\bullet, \cdot)$ is reduced, and so $S$ is a reduced information semialgebra. On the other hand, $0$ is not a limit point of $S^\bullet$, and \cite[Proposition~4.5]{fG19} guarantees that $(S,+)$ is a BFM. It then follows from \cite[Corollary~1.3.3]{GH06} that $(S,+)$ is an ACCP monoid.
	
To verify that $(S^\bullet, \cdot)$ is an ACCP monoid, suppose that $(x_n S^\bullet)_{n \in \nn_0}$ is an ascending chain of principal ideals of $(S^\bullet, \cdot)$. Then there exists a sequence $(y_n)_{n \in \nn}$ of elements of $S^\bullet$ such that $x_n = x_{n+1} y_{n+1}$ for every $n \in \nn_0$. Note that $(x_n)_{n \in \nn_0}$ is a decreasing sequence that converges to some $\ell \in \rr_{\ge 1}$. The fact that $x_0 = x_n \prod_{i=1}^n y_i$ for each $n \in \nn$ implies that $\lim_{n \to \infty} y_n = 1$. Then, after removing finitely many terms from $(x_n S^\bullet)_{n \in \nn_0}$ if necessary, one may assume that $y_n < 2$ for every $n \in \nn$. Observe that $S \cap (0,2) \subseteq \{e^q : q \in M\}$. Thus, for each $n \in \nn$, there exists $q_n \in M$ such that $y_n = e^{q_n}$. Let $x_0 = c_1 e^{r_1} + \dots + c_k e^{r_k}$ for some $k \in \nn$, coefficients $c_1, \dots, c_k \in \nn$, and exponents $r_1, \dots, r_k \in M$ with $r_1 < \dots < r_k$. Since $(S,+)$ is free on $\{e^{q} : q \in M\}$, the equalities $x_0 = x_n \prod_{i=1}^n y_i$ (for all $n \in \nn$) guarantee that $x_n = c_1 e^{r_{n,1}} + \dots + c_k e^{r_{n,k}}$, where $r_{n,1}, \dots, r_{n,k} \in M$ satisfy that $r_{n,j} + \sum_{i=1}^n q_i = r_j$ for every $j \in \ldb 1, k \rdb$. Observe now that $r_{n,j} = r_{n+1,j} + q_{n+1} \in r_{n+1,j} + M$ for each $n \in \nn$ and each $j \in \ldb 1,k \rdb$. As a result, $(r_{n,j} + M)_{n \in \nn}$ is an ascending chain of principal ideals of $M$ for each $j \in \ldb 1,k \rdb$. Since $M$ is an ACCP monoid by \cite[Theorem~5.2]{fG18b}, the ascending chain of principal ideals $(r_{n,j} + M)_{n \in \nn}$ eventually stabilizes. Hence $(x_n S^\bullet)_{n \in \nn_0}$ must eventually stabilize. Thus, $(S^\bullet, \cdot)$ is an ACCP monoid.
	
Finally, we show that $(S^\bullet, \cdot)$ is not a BFM. It is clear that $M$ is not a BFM (for instance, $p (1/p)$ is a length-$p$ factorization in $\mathsf{Z}(1)$ for every $p \in \pp$). On the other hand, $M$ is isomorphic to the multiplicative monoid $N = \{e^q : q \in M\}$ and, therefore, $N$ is not a BFM. As an immediate consequence of the Lindemann-Weierstrass Theorem, one obtains that $N$ is a divisor-closed submonoid of $(S^\bullet, \cdot)$. Hence the monoid $(S^\bullet, \cdot)$ is not a BFM.
\end{example}

\smallskip

%A monoid $M$ satisfies the \emph{Kainrath property} provided that $M$ is a BFM and that $L \in \mathcal{L}(M)$ for all $L \subseteq \mathbb{N}_{\ge 2}$. It follows from Proposition~\ref{prop:algebraic observation of T_n(S)} and Theorems~\ref{thm:BFM characterizations for U_n(S)} and~\ref{thm:BFM characterizations} that if $S$ is a reduced information semialgebra such that $(S,+)$ satisfies the Kainrath property, both $U_n(S)$ and $T_n(S)^\bullet$ satisfy the Kainrath property for every $n \in \nn_{\ge 2}$. However, we have no examples of a bi-BFM reduced information semialgebra $S$ such that $(S,+)$ satisfies the Kainrath property. We conclude this subsection with the following question.
%
%\smallskip
%
%\begin{question}
%Is there a reduced information semialgebra $S$ such that $U_n(S)$ (or $T_n(S)^\bullet$) satisfies the Kainrath property?
%\end{question}
%
%\medskip

%%%%%%%%%%%%%%%%%%%%%%%%%%%%%%
\subsection{The Finite Factorization Property} As we did for the bounded factorization property in the previous subsection, we shall prove in this one that the finite factorization property can be transferred back and forth between a reduced information semialgebra $S$ and the monoids $U_n(S)$ and $T_n(S)^\bullet$ for every $n \ge 2$. We begin by considering $U_n(S)$.

\smallskip

\begin{theorem} \label{thm:FFM characterizations for U_n(S)}
Let $S$ be a reduced information semialgebra. The following statements are equivalent.
\begin{enumerate}
	\item[(a)] $U_n(S)$ is an FFM for every $n \in \nn_{\ge 2}$.
	\smallskip
	
	\item[(b)] $U_n(S)$ is an FFM for some $n \in \nn_{\ge 2}$.
	\smallskip
	
	\item[(c)] $(S,+)$ is an FFM.
\end{enumerate}
\end{theorem}

\begin{proof}
(a) $\Rightarrow$ (b): This is clear.
\smallskip
	
(b) $\Rightarrow$ (c): Since $n \ge 2$, part~3 of Proposition~\ref{prop:algebraic observation of T_n(S)} allows us to identify $(S,+)$ with a divisor-closed submonoid of $U_n(S)$. Since $\mathsf{Z}_S^*(x) = \mathsf{Z}_{T_n(S)^\bullet}^*(x)$ for all $x \in S$, the fact that $U_n(S)$ is an FFM guarantees that $(S,+)$ is also an FFM.
\smallskip
	
(c) $\Rightarrow$ (a): Suppose that $(S,+)$ is an FFM and take $B \in U_n(S) \setminus \{I_n\}$. In order to show that $|\mathsf{Z}^*_{U_n(S)}(B)| < \infty$, we need the following claim.
\smallskip

{\it Claim:} The set $\{A \in \mathcal{A}(U_n(S)) : A \mid_p B \}$ is finite.
\smallskip

{\it Proof of Claim:} Suppose that $A \mid_p B$ for some $A \in \mathcal{A}(U_n(S))$. Then there exist $a \in \mathcal{A}_+(S)$ and $k,\ell \in \ldb 1,n \rdb$ with $k < \ell$ such that $A = I_n + aE_{k \ell}$, and there exist $C,D \in U_n(S)$ such that $B = CAD$. Since $[CA]_{k \ell} = a + C_{k \ell}$, we obtain that $a \mid_S [CA]_{k \ell}$. In addition,
\[
	B_{k \ell} = \sum_{j=1}^n [CA]_{kj} D_{j \ell} = [CA]_{k \ell} + \sum_{j \in \ldb 1,n \rdb \setminus \{\ell\}} [CA]_{kj} D_{j \ell},
\]
which implies that $[CA]_{k \ell} \mid_S B_{k \ell}$. Consequently, $a \mid_S B_{k \ell}$. Since $(S,+)$ is a reduced FFM, the set $\mathsf{Z}_{S}(B_{k \ell})$ is finite, and it follows from Lemma~\ref{lem:A(x) is finite iff Z(x) is finite} that $B_{k \ell}$ is divisible in $S$ by only finitely many atoms. This in turn implies that
\[
	|\{A \in \mathcal{A}(U_n(S)) : A \mid_p B \}| \le |\{(a,(i,j)) \in \mathcal{A}(S) \times \ldb 1,n \rdb^2 : i < j \text{ and } a \mid_S B_{i,j} \}| < \infty,
\]
from which the claim follows.

\smallskip
	
Since $(S,+)$ is an FFM, it must be also a BFM and, by Theorem~\ref{thm:BFM characterizations for U_n(S)}, we see that $U_n(S)$ is a BFM. Thus, $\mathsf{L}_{U_n(S)}(B)$ is finite. This, along with the fact that there are only finitely many atoms in $U_n(S)$ that divide $B$ up to permutation, implies that $|\mathsf{Z}^*_{U_n(S)}(B)| < \infty$. Hence $U_n(S)$ is an FFM.
\end{proof}

\smallskip

From Theorem~\ref{thm:FFM characterizations for U_n(S)}, we immediately have the following corollary.

\smallskip

\begin{cor} \label{cor:FFM implies BFM for U_n(S)}
Let $S$ be a reduced information semialgebra. If $U_n(S)$ is an FFM, then $U_n(S)$ must be a BFM.
\end{cor}

\smallskip

Using Theorem~\ref{thm:FFM characterizations for U_n(S)} and Corollary~\ref{cor:FFM implies BFM for U_n(S)} we now extend Diagram~(\ref{diag:atomic diagram with BFM included for U_n(S)}) one step further. For each reduced information semialgebra $S$ and every $n \in \nn_{\ge 2}$, each implication in the next diagram holds.

\smallskip

\begin{equation} \label{diag:full atomic diagram for U_n(S)}
	\begin{tikzcd}
		(S,+) \text{ is \textbf{FFM}} \arrow[r, Rightarrow] \arrow[d, Leftrightarrow] & (S,+) \text{ is \textbf{BFM}} \arrow[r, Rightarrow] \arrow[d, Leftrightarrow] &  (S,+) \text{ is \textbf{ACCP}} \arrow[r, Rightarrow] \arrow[d, Leftrightarrow] & (S,+) \text{ is \textbf{atomic}} \arrow[d, Leftrightarrow] \\
		U_n(S) \text{ is an \textbf{FFM}} \arrow[r, Rightarrow] & U_n(S) \text{ is a \textbf{BFM}} \arrow[r, Rightarrow] & U_n(S) \text{ is \textbf{ACCP}} \arrow[r, Rightarrow] & U_n(S) \text{ is \textbf{atomic}}
	\end{tikzcd}
\end{equation}

\smallskip

We proceed to establish a result analogous to that of Theorem~\ref{thm:FFM characterizations for U_n(S)} for the monoids $T_n(S)^\bullet$.

\smallskip

\begin{theorem} \label{thm:FFM characterizations}
Let $S$ be a reduced information semialgebra. The following statements are equivalent.
\begin{enumerate}
	\item[(a)] $T_n(S)^\bullet$ is an FFM for every $n \in \nn_{\ge 2}$.
	\smallskip

	\item[(b)] $T_n(S)^\bullet$ is an FFM for some $n \in \nn_{\ge 2}$.
	\smallskip

	\item[(c)] $S$ is a bi-FFM.
\end{enumerate}
\end{theorem}

\begin{proof}
(a) $\Rightarrow$ (b): This is clear.
\smallskip
	
(b) $\Rightarrow$ (c): First, identify both $(S,+)$ and $(S^\bullet, \cdot)$ with divisor-closed submonoids of $T_n(S)^\bullet$ using part~2 and part~3 of Proposition~\ref{prop:algebraic observation of T_n(S)}, and then fix $x \in S$. Since $(S,+)$ is a divisor-closed submonoid of $T_n(S)^\bullet$, the equality $\mathcal{A}(S) = \mathcal{A}(T_n(S)^\bullet) \cap S$ holds. Consequently, we see that $\mathsf{Z}_S^*(x) = \mathsf{Z}_{T_n(S)^\bullet}^*(x)$. Therefore the fact that $T_n(S)^\bullet$ is an FFM implies that $(S,+)$ is an FFM as well. A similar argument shows that $(S^\bullet, \cdot)$ is an FFM.
\smallskip
	
(c) $\Rightarrow$ (a): Fix $n \in \nn_{\ge 2}$, and then take $B \in T_n(S)^\bullet$. As we did in the corresponding part of Theorem~\ref{thm:FFM characterizations for U_n(S)}, we shall verify that $B$ has only finitely many rigid factorizations in $T_n(S)^\bullet$ by arguing that there are only finitely many atoms of $T_n(S)^\bullet$ that are rigid divisors of $B$. To do so, take $A \in \mathcal{A}(T_n(S)^\bullet)$ such that $A \mid_p B$.
	
First, suppose that $A$ is an atom of multiplicative type. In this case, $A = I_n + (a-1)E_{ii}$ for some $i \in \ldb 1,n \rdb$ and $a \in \mathcal{A}(S^\bullet)$. As $A \mid_p B$, it follows that $a \mid_{S^\bullet} \det B$ and thus $a \in D_{S^\bullet}(\det B)$. Since $(S^\bullet, \cdot)$ is an FFM the set $\mathsf{Z}_{S^\bullet}(\det B)$ is finite, and since $(S^\bullet, \cdot)$ is a reduced monoid Lemma~\ref{lem:A(x) is finite iff Z(x) is finite} guarantees that $|A_{S^\bullet} (\det B)| < \infty$. Thus, letting $A_\times(B)$ denote the set of atoms of $T_n(S)^\bullet$ of multiplicative type that happen to be rigid divisors of $B$, we obtain
\begin{equation} \label{eq:multiplicative atoms FFM}
	|A_\times(B)| \le \big| \{(a,i) \in \mathcal{A}(S^\bullet) \times \ldb 1,n \rdb : a \mid_{S^\bullet} \det B \} \big|  =  \big| A_{S^\bullet}(\det B) \times \ldb 1,n \rdb \big| < \infty.
\end{equation}
	
Now suppose that $A$ is an atom of additive type. Take $a \in \mathcal{A}_+(S)$ and $k, \ell \in \ldb 1,n \rdb$ with $k < \ell$ such that $A = I_n + aE_{k \ell}$. Writing $B = CAD$ for some $C, D \in T_n(S)^\bullet$, one sees that $[CA]_{k \ell} = C_{kk}a + C_{k \ell}$ and so
\[
	B_{k \ell} = \sum_{j=1}^n [CA]_{kj} D_{j \ell} = C_{kk}D_{\ell \ell}a + C_{k \ell}D_{\ell \ell} + \sum_{j \in \ldb 1,n \rdb \setminus \{\ell\}} [CA]_{kj} D_{j \ell}.
\]
Then there exists $a' \in S^\bullet$ such that $aa' \mid_S B_{k \ell}$, i.e., $aa' \in D_S(B_{k \ell})$, where $D_S(B_{k \ell})$ denotes the set consisting of all divisors of $B_{k \ell}$ in $(S,+)$. Since $(S,+)$ is an FFM the set $\mathsf{Z}_{S}(B_{k \ell})$ is finite, and since $(S,+)$ is a reduced monoid it follows from Lemma~\ref{lem:A(x) is finite iff Z(x) is finite} that $D_S(B_{k \ell})$ is also finite. In addition, $aa' \in D_S(B_{k \ell})$ implies that
\begin{equation} \label{eq:FFM first equation}
	a \in \bigcup_{d \in D_S(B_{k \ell})} D_{S^\bullet}(d),
\end{equation}
Once again, since $(S^\bullet, \cdot)$ is a reduced FFM, Lemma~\ref{lem:A(x) is finite iff Z(x) is finite} implies that $D_{S^\bullet}(d)$ is a finite set for each $d \in D_S(B_{k \ell})$. Therefore letting $A_+(B)$ denote the set of all atoms of $\mathcal{A}(T_n(S)^\bullet)$ of additive type that are rigid divisors of $B$ in $T_n(S)^\bullet$, it follows from~(\ref{eq:FFM first equation}) that
\begin{equation} \label{eq:additive atoms FFM}
	|A_+(B)| \le \bigg{|} \bigcup_{1 \le i < j \le n} \ \bigcup_{d \in D_S(B_{i j})} D_{S^\bullet}(d) \bigg{|} < \infty.
\end{equation}
\smallskip
	
Putting together~(\ref{eq:multiplicative atoms FFM}) and~(\ref{eq:additive atoms FFM}), we see that $|A_\times(B) \cup A_+(B)| < \infty$. This, along with the fact that $T_n(S)^\bullet$ is a BFM (by Theorem~\ref{thm:BFM characterizations}), immediately implies that $T_n(S)^\bullet$ is indeed an FFM, which concludes the proof.
\end{proof}

\smallskip

From Theorem~\ref{thm:FFM characterizations} we deduce the following corollary.

\smallskip

\begin{cor} \label{cor:FFM implies BFM for T_n(S)}
	Let $S$ be a reduced information semialgebra. If $T_n(S)^\bullet$ is an FFM, then $T_n(S)^\bullet$ is a BFM.
\end{cor}

\smallskip

In virtue of Theorem~\ref{thm:FFM characterizations} and Corollary~\ref{cor:FFM implies BFM for T_n(S)}, we can extend Diagram~(\ref{diag:atomic diagram with BFM included}). For each reduced information semialgebra $S$ and for every $n \in \nn_{\ge 2}$, each implication in the following diagram holds.
\begin{equation} \label{diag:full atomic diagram}
	\begin{tikzcd}
		S \text{ is \textbf{bi-FFM}} \arrow[r, Rightarrow] \arrow[d, Leftrightarrow] & S \text{ is \textbf{bi-BFM}} \arrow[r, Rightarrow] \arrow[d, Leftrightarrow] &  S \text{ is \textbf{bi-ACCP}} \arrow[r, Rightarrow] \arrow[d, Leftrightarrow] & S \text{ is \textbf{bi-atomic}} \arrow[d, Leftrightarrow] \\
		T_n(S)^\bullet \text{ is an \textbf{FFM}} \arrow[r, Rightarrow] & T_n(S)^\bullet \text{ is a \textbf{BFM}} \arrow[r, Rightarrow] & T_n(S)^\bullet \text{ is \textbf{ACCP}} \arrow[r, Rightarrow] & T_n(S)^\bullet \text{ is \textbf{atomic}}
	\end{tikzcd}
\end{equation}

\smallskip

The leftmost horizontal implications in Diagram~\eqref{diag:full atomic diagram} are not, in general, reversible. To verify this, we now exhibit a reduced information semialgebra $S$ such that, for every $n \in \nn_{\ge 2}$, the monoid $T_n(S)^\bullet$ is a BFM but not an FFM.

\smallskip

\begin{example}
Consider the Puiseux conductive information semialgebra $Q_2$. It is reduced since $1 = \inf Q_2^\bullet$. Clearly, $\ln Q_2^\bullet := \{\ln q : q \in Q_2^\bullet\}$ is an additive submonoid of $(\rr_{\ge 0},+)$. Since $0$ is not a limit point of either $Q_2^\bullet$ or $\ln Q_2^\bullet \setminus \{0\}$, it follows from \cite[Proposition~4.5]{fG19} that the additive monoids~$Q_2$ and $\ln Q_2^\bullet$ are BFMs. Thus, $Q_2$ is a bi-BFM. As a result, Theorem~\ref{thm:BFM characterizations for U_n(S)} and Theorem~\ref{thm:BFM characterizations} ensure that $U_n(Q_2)$ and $T_n(Q_2)^\bullet$ are BFMs for every $n \in \nn_{\ge 2}$. By Theorems~\ref{thm:FFM characterizations for U_n(S)} and~\ref{thm:FFM characterizations}, showing that none of the monoids $U_n(Q_2)$ and $T_n(Q_2)^\bullet$ (for $n \in \nn_2$) is an FFM amounts to verifying that the monoid $(Q_2,+)$ is not an FFM. By Proposition~\ref{prop:atomicity of elementary conductive PIA} the equalities $\mathcal{A}_+(Q_2) = (2,3) \cap \qq$ and $\mathcal{A}_\times(Q_2) = [2,4) \cap \qq$ hold. For each $x \in (4,5) \cap \qq$ the formal sum $(2 + 1/n) + (x - 2 - 1/n)$ is a length-$2$ factorization of $x$ in $(Q_2,+)$ for every integer $n > \frac{1}{x-4}$. Hence $|\mathsf{Z}_{Q_2}(x)| = \infty$, which implies that the monoid $(Q_2,+)$ is not an FFM. % In a similar manner, one can show that $(Q^\bullet, \cdot)$ is not an FFM. Consequently, $Q_r$ is a bi-BFM that is not a bi-FFM.
\end{example}

\medskip

%%%%%%%%%%%%%%%%%%%
\subsection{Half-Factoriality} We have seen that many important atomic factorization-theoretic properties transfer back and forth between a reduced information semialgebra $S$ and the noncommutative multiplicative monoids $U_n(S)$ and $T_n(S)$ (for every $n \in \nn_{\ge 2}$). However, there are arithmetic properties that fail to transfer from $S$ to $T_n(S)^\bullet$ (or $U_n(S)$) even when $S = \nn_0$. For instance, $(\mathbb \nn_0,+)$ and $(\mathbb \nn, \cdot)$ are both UFMs, and yet $T_n(\mathbb N_0)^\bullet$ has full infinite elasticity~\cite[Theorem~3.11]{BS20}. Continuing in this direction, we will see, as a consequence of Proposition~\ref{prop:T_n(S) is almost never an HFM}, that for every $n \in \nn_{\ge 2}$ the monoid $T_n(S)^\bullet$ is not \emph{half-factorial} even if both monoids $(S,+)$ and $(S^\bullet, \cdot)$ are.

An atomic monoid $M$ is called an \emph{HFM} (or a \emph{half-factorial monoid}) if for all $x \in M \setminus U(M)$ whenever $x = a_1 \dots a_m = b_1 \dots b_n$ for $m,n \in \nn$ and $a_1, \dots, a_m, b_1, \dots, b_n \in \mathcal{A}(M)$, the equality $m = n$ holds. The half-factorial property has been well-studied in the category of commutative monoids and, in that setting,  it is well-known that each of the implications in Diagram~(\ref{equation:anderson}) hold (and are not, in general, reversible).

However, as we now observe, the monoids $T_n(S)^\bullet$ are almost never HFMs and, as a result, we cannot extend the part of Diagram~\eqref{equation:anderson} involving HFMs from the commutative setting to the setting of upper triangular matrices over reduced information semialgebras.

\smallskip

\begin{prop} \label{prop:T_n(S) is almost never an HFM}
Let $S$ be a reduced information semialgebra. Then $T_n(S)^\bullet$ is not an HFM for any $n \ge 2$.
\end{prop}

\begin{proof}
It follows from part~3 of Proposition~\ref{prop:algebraic observation of T_n(S)} that when $(S,+)$ is not atomic, the monoid $T_n(S)^\bullet$ is not atomic and, in particular, $T_n(S)^\bullet$ is not an HFM.
	
Assume now that $(S,+)$ is atomic. In this case, we claim that $1 \in \mathcal A_+(S)$. Suppose for a contradiction that this is not the case. Writing $1 = x+y$ for some $x,y \in S^\bullet$, one observes that $s = s(x+y) = sx + sy \in S^\bullet + S^\bullet$ for all $s \in S^\bullet$. Thus, $(S,+)$ must be antimatter. However, only groups can be atomic and antimatter simultaneously, and $(S,+)$ is not a group. Then $1 \in \mathcal{A}_+(S)$. Now fix a divisor-closed submonoid $T_2$ of $T_n(S)^\bullet$ isomorphic to $T_2(S)^\bullet$. Since $1 \in \mathcal{A}_+(S)$, the matrix $\begin{psmallmatrix} 1 & 1 \\ 0 & 1\end{psmallmatrix}$ is an atom of $T_2$. Now for every $m \in \mathbb N_{\geq 2}\subseteq S$, the equality
\[
	A := \begin{pmatrix} 1 & m+1 \\ 0 & m \end{pmatrix} = \begin{pmatrix} 1 & 0 \\ 0 & m \end{pmatrix} \begin{pmatrix} 1 & 1 \\ 0 & 1 \end{pmatrix}^m = \begin{pmatrix} 1 & 1 \\ 0 & 1 \end{pmatrix} \begin{pmatrix} 1 & 0 \\ 0 & m \end{pmatrix}
\]
holds in $T_2$. Thus, for any $\ell \in \mathsf L_{T_2} \left(\begin{psmallmatrix} 1 & 0 \\ 0 & m \end{psmallmatrix}\right)$, one obtains that $\ell+1, \ell+m \in \mathsf{L}_{T_2}(A)$. This implies that~$T_2$ is not an HFM. Hence $T_n(S)^\bullet$ cannot be an HFM.
\end{proof}

\smallskip

We conclude with the following corollary. 

\smallskip

\begin{cor}
Let $S$ be a reduced information semialgebra. Then $T_n(S)^\bullet$ is not a rigid UFM for any $n \ge 2$.
\end{cor}


\bigskip

\begin{thebibliography}{20}
	
	\bibitem{AAZ90} D.~D. Anderson, D.~F. Anderson, and M.~Zafrullah: \emph{Factorizations in integral domains}, J. Pure Appl. Algebra {\bf 69} (1990), 1--19.
	
	 \bibitem{BBG} D. Bachman, N.~R. Baeth, and J. Gossell: \emph{Factorizations of upper triangular matrices}, Linear Algebra Appl. \textbf{450} (2014), 138--157.
	 
	 \bibitem{BJ16} N.~R. Baeth and J. Jeffries: \emph{Factorizations of block triangular matrices}, Linear Algebra Appl. \textbf{511} (2016), 403--420.
	 
	\bibitem{BS20} N.~R. Baeth and R.~Sampson: \emph{Upper triangular matrices over information algebras}, Linear Algebra Appl. \textbf{587} (2020), 334--357.
	
	\bibitem{BS15} N.~R. Baeth and D. Smertnig: \emph{Factorization theory: From commutative to noncommutative settings}, J. Algebra \textbf{441} (2015), 475--551.
	
	\bibitem{aB90} A. Baker: \emph{Transcendental Number Theory} (2nd ed.), Cambridge Mathematical Library, Cambridge University Press, 1990.
	
	\bibitem{BHL17} J. Bell, A. Heinle, and V. Levandovskyy: \emph{On noncommutative finite factorization domains}, Trans. Amer. Math. Soc. \textbf{369} (2017), 2675--2695.
	
	\bibitem{CGG19} S.~T. Chapman, F. Gotti, and M. Gotti, \emph{Factorization invariants of Puiseux monoids generated by geometric sequences}, Comm. Algebra \textbf{48} (2020), 380--396.
	
	\bibitem{CGG20} S.~T. Chapman, F. Gotti, and M. Gotti, \emph{When is a Puiseux monoid atomic?}, Amer. Math. Monthly (to appear).
	
	\bibitem{CZL15} Y. Chen, X. Zhao, and Z. Liu: \emph{On upper triangular nonnegative matrices}, Czechoslovak Math. J. \textbf{65} (2015), 1--20.
	
	\bibitem{pC85} P.~M. Cohn: \emph{Free Rings and Their Relations} (2nd ed.), London Mathematical Society Monographs, vol. 19, Academic Press Inc., London, 1985.
	
	\bibitem{CS03} J.~H. Conway and D.~A. Smith: \emph{On Quaternions and Octonions: Their Geometry, Arithmetic, and Symmetry}, A K Peters Ltd., Natick, MA, 2003.
	
	%\bibitem{Fu70} L.~Fuchs: \emph{Infinite Abelian Groups~I}, Academic Press, 1970.
	
	%\bibitem{Fu73} L. Fuchs: \emph{Infinite Abelian Groups~II}, Academic Press, 1973.
	
	\bibitem{aG13} A. Geroldinger: \emph{Non-commutative Krull monoids: A divisor theoretic approach and their arithmetic}, Osaka J. Math. \textbf{50} (2013), 503--539.
	
	\bibitem{GGT19} A. Geroldinger, F. Gotti, and S. Tringali: \emph{On strongly primary monoids, with a focus on Puiseux monoids}. Preprint available on arXiv: https://arxiv.org/pdf/1910.10270.pdf
	
	\bibitem{GH06} A. Geroldinger and F. Halter-Koch: \emph{Non-unique Factorizations: Algebraic, Combinatorial and Analytic Theory}, Pure and Applied Mathematics Vol. 278, Chapman \& Hall/CRC, Boca Raton, 2006.
	
	\bibitem{jG99} J.~S. Golan: \emph{Semirings and their Applications}, Kluwer Academic Publishers, 1999.
	
	\bibitem{fG18b} F. Gotti: \emph{Atomic and antimatter semigroup algebras with rational exponents}. Available on arXiv: https://arxiv.org/abs/1801.06779.
	
	\bibitem{fG19} F. Gotti: \emph{Increasing positive monoids of ordered fields are FF-monoids}, J. Algebra \textbf{518} (2019), 40--56.
	
	\bibitem{fG17} F. Gotti: \emph{On the atomic structure of Puiseux monoids}, J. Algebra Appl. \textbf{16} (2017), 1750126.
	
	\bibitem{fG18} F. Gotti: \emph{Puiseux monoids and transfer homomorphisms}, J. Algebra \textbf{516} (2018), 95--114.
	
	\bibitem{GG17} F.~Gotti and M.~Gotti: \emph{Atomicity and boundedness of monotone Puiseux monoids}, Semigroup Forum {\bf 96} (2018), 536--552.
	
	\bibitem{pG01} P.~A.~Grillet: \emph{Commutative Semigroups}, Advances in Mathematics, vol.~2, Kluwer Academic Publishers, Boston, 2001.
	
	\bibitem{jK72} J. Kuntzmann: \emph{Th\'eorie des R\'eseaux (Graphes)}, Dunod, Paris, 1972.
	
	\bibitem{dS13} D. Smertnig: \emph{Sets of lengths in maximal orders in central simple algebras}, J. Algebra \textbf{390} (2013), 1--43.

\end{thebibliography}
\end{document}